\documentclass{article}

\usepackage{PRIMEarxiv}

\usepackage[utf8]{inputenc} 
\usepackage[T1]{fontenc}    
\usepackage{hyperref}       
\usepackage{url}            
\usepackage{booktabs}       
\usepackage{amsfonts}       
\usepackage{nicefrac}       
\usepackage{microtype}      
\usepackage{lipsum}
\usepackage{fancyhdr}       
\usepackage{graphicx}       
\usepackage{amsmath,amsfonts}
\usepackage{algorithmic}
\usepackage{algorithm}
\usepackage{array}
\usepackage[caption=false,font=normalsize,labelfont=sf,textfont=sf]{subfig}
\usepackage{textcomp}
\usepackage{stfloats}
\usepackage{url}
\usepackage{verbatim}
\usepackage{subcaption}
\usepackage{booktabs}
\usepackage{multirow}
\usepackage{amsthm}
\usepackage{amsmath}
\usepackage{makecell}
\newtheorem{theorem}{Theorem}
\newtheorem{lemma}{Lemma}

\usepackage{xcolor}
\graphicspath{{media/}}     

\title{Variable Projection Algorithms: Theoretical Insights and A Novel Approach for Problems with Large Residual}

\author{
  Guangyong Chen, Peng Xue, Min Gan, Wenzhong Guo \\
  College of Computer and Data Science \\
  Fuzhou University \\
  China\\
  \texttt{cgykeda@mail.ustc.edu.cn, xuep802@gmail.com, ganmin@ieee.org, guowenzhong@fzu.edu.cn} \\
   \And
  Jing Chen \\
  School of Science \\
  Jiangnan University \\
  China\\
  \texttt{chenjing1981929@126.com} \\
     \And
  C. L. Philip. Chen \\
  School of Computer Science and Engineering \\
  South China University of Technology \\
  China\\
  \texttt{Philip.Chen@ieee.org} \\
}

\begin{document}
\maketitle

\begin{abstract}
This paper delves into an in-depth exploration of the Variable Projection (VP) algorithm, a powerful tool for solving separable nonlinear optimization problems across multiple domains, including system identification, image processing, and machine learning. We first establish a theoretical framework to examine the effect of the approximate treatment of the coupling relationship among parameters on the local convergence of the VP algorithm and theoretically prove that the Kaufman’s VP algorithm can achieve a similar convergence rate as the Golub \& Pereyra’s form. These studies fill the gap in the existing convergence theory analysis, and provide a solid foundation for understanding the mechanism of VP algorithm and broadening its application horizons. Furthermore, inspired by these theoretical insights, we design a refined VP algorithm, termed VPLR, to address separable nonlinear optimization problems with large residual. This algorithm enhances convergence performance by addressing the coupling relationship between parameters in separable models and continually refining the approximated Hessian matrix to counteract the influence of large residual. The effectiveness of this refined algorithm is corroborated through numerical {experiments}.
\end{abstract}

\section{Introduction}
\label{sec:introduction}
Separable nonlinear models, which are ubiquitous in disciplines such as system identification \cite{saraf2022efficient,chen2022modified,li2021distributed}, image processing \cite{espanol2023variable,lin2017robust,sheng2015constrained,chung2023hybrid}
, signal processing 
\cite{kovacs2019generalized,barbieri1992two}
, and neural networks \cite{kovacs2022vpnet,lederer2022cooperative,newman2021train} are constituted by a linear combination of nonlinear basis functions. Conventionally, they can be delineated as:
\begin{equation}\label{eq:1.1}
f(\boldsymbol{a},\boldsymbol{c};\boldsymbol{x}_i)=\sum\limits_{j=1}^n c_j\phi_j(\boldsymbol{a};\boldsymbol{x}_i)+\epsilon_i,
\end{equation}
where $\boldsymbol c=(c_1,\cdots,c_n)^{\top}$ represents the linear parameter of the model; $\boldsymbol{a}\in \mathcal{R}^k$ is the nonlinear parameter; $\phi_j(\boldsymbol{a};\boldsymbol{x}_i)$ is a nonlinear basis function; $\boldsymbol{x}_i$ is the vector intrinsically associated with the state in the model; and $\epsilon_i$ denotes the noise. Given $m$ pairs of {observed} data $\{(\boldsymbol{x}_i,\boldsymbol{y}_i)\}_{i=1}^m$, the identification of model \eqref{eq:1.1} can be \textcolor{black}{formulated} as an optimization problem:
\begin{equation}\label{eq:1.2}
    \min\limits_{\boldsymbol{a},\boldsymbol{c}} r(\boldsymbol{a},\boldsymbol{c})=\frac{1}{2}\sum\limits_{i=1}^m(y_i-\sum\limits_{j=1}^nc_j\phi_j(\boldsymbol{a};\boldsymbol{x}_i))^2.
\end{equation}
The optimization objective delineated above can be \textcolor{black}{compactly represented} in matrix form as follows:
\begin{equation}\label{eq:1.3}
    \min\limits_{\boldsymbol{a},\boldsymbol{c}} r(\boldsymbol{a},\boldsymbol{c})=\frac{1}{2}\left\|\boldsymbol{y}-\mathbf{\Phi}(\boldsymbol{a})\boldsymbol{c}\right\|^2,
\end{equation}
where $\boldsymbol{y}=(y_1,\cdots,y_m)^{\top}$ signifies the observation vector, and $\left\|\cdot\right\|$ denotes the $L_2$-norm. \textcolor{black}{Unless stated otherwise,} we also use it to refer to the matrix norm induced by the vector $L_2$ norm.

Optimization problem \eqref{eq:1.3} is {referred to as the} \textit{Separable Non-Linear Least Squares} (SNLLS) problem \cite{golub2003separable}. Such problems \textcolor{black}{appear in various forms across} different application domains. For instance, in the field of computer vision, the issue of 3D trajectory reconstruction with missing data \cite{tomasi1992shape,hong2015secrets,hu2023sar} can be formulated as the following optimization problem:
\[
\min_{\mathbf U,\mathbf V}\|\mathbf W \odot (\mathbf{Y} - \mathbf{UV})\|_{\text F}^2,
\]
where \(\mathbf W \) is the marking matrix for missing data, \( \mathbf{Y} \) is the observation matrix, \( \mathbf U \) and \( \mathbf V \) are the low-rank matrices, $\odot$ represents element-wise multiplication, and \(\left\|\cdot\right\|_{\text{F}}\) denotes the Frobenius norm. {In the context of data analysis, the commonly used} sparse principal component analysis \cite{erichson2020sparse,dorabiala2024ensemble,chen2023variable} can be summarized as:
\[\begin{split}
\min_{\mathbf U,\mathbf V}\  &g(\mathbf U,\mathbf V) = \frac{1}{2} \|\mathbf{Y} - \mathbf{Y}\mathbf V\mathbf{U}^\top\|^2_{\text F} + \Psi(\mathbf V) \\& \text{s.t.,} \quad \mathbf{U}^\top \mathbf{U} = \mathbf I.
\end{split}
\]
where \(\Psi(\cdot)\) denotes the sparse regularizer. The resolution of such optimization problem involves various algorithms, including gradient-based optimization methods \cite{victor2022system}, alternating optimization \cite{khatana2023dc,hardt2014understanding,yang2022proximal,kong2021kalman}, joint optimization \cite{1467459}, Wiberg algorithm \cite{wiberg1976computation,okatani2007wiberg}, Majorization-Minimization algorithm \cite{lin2017robust,landeros2023mm}, and the Variable Projection (VP) algorithm \cite{golub2003separable,espanol2023variable,van2021variable}. Due to their inherently non-convex and non-smooth nature, these problems remain \textcolor{black}{inherently} challenging to solve. 

By focusing on the special structure {inherent in} these problems, more efficient algorithms can be \textcolor{black}{developed}. \textcolor{black}{ One such representative algorithm is the Variale Projection (VP) algorithm, originally proposed by Golub \& Pereyra \cite{golub2003separable,golub1973differentiation}.} The VP algorithm partitions the parameters into linear parameters $\boldsymbol{c}$ and nonlinear parameters $\boldsymbol{a}$. For fixed values of $\boldsymbol{a}$, $\boldsymbol{c}$ can be articulated as a function of {$\boldsymbol{a}$} by solving the linear least squares problem, as follows:
\begin{equation}
    \hat{\boldsymbol{c}}{\color{black}(\boldsymbol{a})}=\arg\min\limits_{\boldsymbol{c}}\frac{1}{2}\left\|\boldsymbol{y}-\mathbf{\Phi}(\boldsymbol{a})\boldsymbol{c}\right\|^2=\mathbf{\Phi}^\dagger{\color{black}(\boldsymbol{a})}\boldsymbol{y}.\label{2-1}
\end{equation}
where $\mathbf{\Phi}^\dagger(\boldsymbol{a})$ is the Moore-Penrose inverse of $\mathbf{\Phi}(\boldsymbol{a})$. Substituting \eqref{2-1} into \eqref{eq:1.3} results in a reduced function that only {contains} the nonlinear parameters $\boldsymbol{a}$:
\begin{equation}
 \min\limits_{\boldsymbol{a}}{\color{black} R_2}(\boldsymbol{a})=\frac{1}{2}{\boldsymbol{r}_2}^{\top}\boldsymbol{r}_2=\frac{1}{2}\left\|\boldsymbol{y}-\mathbf{\Phi}{\color{black}(\boldsymbol{a})}\mathbf{\Phi}^\dagger{\color{black}(\boldsymbol{a})}\boldsymbol{y}\right\|^2,\label{2-2}
\end{equation}
where $\boldsymbol{r}_2=(\mathbf{I}-\mathbf{\Phi}(\boldsymbol{a})\mathbf{\Phi}^\dagger(\boldsymbol{a}))\boldsymbol{y}$
is the residual vector. 

The VP algorithm, by solving a convex subproblem, expresses linear parameters as a function of nonlinear parameters, thereby reducing the dimensionality of the parameters. This {enables efficient optimization} in a lower-dimensional parameter space \cite{golub2003separable,zhang2020offline}{, mitigates the ill-conditioning often present in the original problem} \cite{sjoberg1997separable,chen2021insights}, and alleviates the risk of the algorithm oscillating within narrow, elongated valleys or succumbing to sharp local optima \cite{su2023nonmonotone,chen2023variable}. Owing to its efficiency, the VP algorithm has been extensively applied in various fields such as system identification\cite{zeng2017regularized}, signal processing \cite{laadjal2018online,bock2021ecg}, {bundle adjustment \cite{weber2024power,iglesias2023expose}}, and network training \cite{kim2008training,newman2021train}, especially after the publication of comprehensive reviews on the VP algorithm \cite{golub2003separable}, which brought it to the attention of more researchers. 

The success of the VP algorithm hinges on the adept handling of the couplings among the parameters within the model. This pivotal aspect {lies in} the computation of the Jacobian matrix of the reduced objective function \eqref{2-2}.  However, the structure of the reduced function does not necessarily render it less complex than the original objective, especially when designing second-order algorithms, \textcolor{black}{as} computing its Jacobian matrix can be challenging. Fortunately, for SNLLS problems,
Golub \& Pereyara have successfully derived the analytical expression for the Jacobian matrix of the reduced objective:
\begin{equation}
    \mathbf{J}_{\text{GP}}=-\mathbf{P}_{\boldsymbol{\Phi}}^{\perp}D\boldsymbol{\Phi}\boldsymbol{\Phi}^{\dagger}\boldsymbol{y}-(\mathbf{P}_{\boldsymbol{\Phi}}^{\perp}D\boldsymbol{\Phi}\boldsymbol{\Phi}^{\dagger})^\top\boldsymbol{y},\label{jac1}
\end{equation}
where $D\mathbf{\Phi}$ is the Fr\'echet derivative of $\mathbf{\Phi}$, and $\mathbf{P}_\mathbf{\Phi}^\perp=\mathbf{I}-\mathbf{\Phi}\mathbf{\Phi}^\dagger$ is a projection operator that projects a vector into the {orthogonal} complement of the column space of $\mathbf{\Phi}$. \textcolor{black}{For notational simplicity, we denote \( \mathbf{\Phi}(\boldsymbol{a}) \) as \( \mathbf{\Phi} \), omitting the explicit dependence on \(\boldsymbol{a}\)}.

In the realm of practical applications, the task of tackling the intricate coupling relationships among model parameters often presents a complex challenge. To mitigate this, scholars have proposed different simplified forms for the Jacobian matrix. Notably, Kaufman \cite{kaufman1975variable} {introduced} a simplified form under the premise that the second term of the Jacobian matrix \eqref{jac1} has minimal influence on the final outcome: 
\begin{equation}
    \mathbf{J}_{\text{Kau}} = -\mathbf{P}_{\boldsymbol{\Phi}}^{\perp}D\boldsymbol{\Phi}\boldsymbol{\Phi}^{\dagger}\boldsymbol{y}.\label{jac2}
\end{equation}
In \cite{ruano1991new}, Ruano et al. proposed an even more concise form of the Jacobian matrix:
\begin{equation}
    \mathbf{J}_{\text{R}} = -D\boldsymbol{\Phi}\boldsymbol{\Phi}^{\dagger}\boldsymbol{y}.\label{jac3}
\end{equation}
Additionally, Song et al. \cite{song2020secant} have constructed an approximate Jacobian matrix using the secant method and proved the convergence of the secant method-based VP algorithm. Shearer et al. \cite{shearer2013generalization} proposed a method for addressing the approximate coupling relationship between parameters in non-least squares situations. 

The various forms of the Jacobian matrix essentially represent distinct approximations of the coupling between parameters. For example, Chen et al. \cite{chen2021insights} have noted that the Kaufman form of the Jacobian matrix fundamentally employs a first-order approximation when computing the derivative of the linear parameters $\boldsymbol{c}$ with respect to the nonlinear parameters $\boldsymbol{a}$. However, a comprehensive theoretical analysis of the effects of these different approximations (i.e., the various forms of the Jacobian matrix) on the performance of the VP algorithm is conspicuously absent in the existing literature. To the best of our knowledge, only Ruhe \& Wedin \cite{ruhe1980algorithms} compared the convergence of VP algorithms based on different forms of Jacobian matrices from the perspective of asymptotical convergence rates, but their comparison was limited to the spectral radius of the Hessian matrix and applicable only to single-step iterations. A theoretical investigation of this issue is crucial for enhancing our understanding of the intrinsic mechanisms of the VP algorithm and for guiding the development of tailored VP algorithms for a \textcolor{black}{broader} range of separable problems. This paper aims to fill this gap. We first provide a framework for analyzing the impact of approximation processing of the coupling relationship between parameters (i.e., using different forms of approximated Jacobian matrix) on the local convergence of the VP algorithm, and theoretically validates the effectiveness of the Kaufman's simplified VP algorithm, demonstrating that it can achieve a local convergence rate similar to that of the Golub \& Pereyra's form. This {study} provides researchers with a unique perspective for understanding the mechanism of VP algorithm and expanding its application in complex scenarios.

Inspired by the theoretical analysis of the convergence \textcolor{black}{behavior} of VP algorithm, this study investigates a class of separable nonlinear optimization problems with large residual. In such cases, the Hessian matrix, obtained from the Jacobian matrix of the reduced function, exhibits significant deviations, which hinder the convergence of the algorithm. To address \textcolor{black}{this challenge}, we introduce an enhanced VP algorithm, termed VPLR. This innovative approach improves the algorithm’s convergence performance by considering the coupling relationships between model parameters, and recursively refining the approximated Hessian matrix to compensate for the effect of large residual. The effectiveness of this enhanced VP algorithm is validated through various numerical experiments.

The main contributions of this paper are summarized as follows:
\begin{itemize}
\item{We first establish a theoretical framework to scrutinize the influence of the approximate treatment of the coupling relationship among parameters on the local convergence of the VP algorithm; and then theoretically validate the efficacy of the Kaufman's form of the VP algorithm, showing that it can achieve a local convergence rate similar to that of the Golub \& Pereyra VP algorithm. This study also offers a unique perspective for to comprehend the VP algorithm, fills the theoretical gap, and suggests new avenues for its application in a wide range of separable nonlinear optimization problems.}

\item {Building on our theoretical analysis, we {propose} an enhanced VP algorithm, termed as VPLR, {specifically designed to address separable nonlinear optimization problems with large residual.}  The algorithm takes into account the coupling relationships among the parameters, and recursively adjusts the Hessian matrix obtained from the Jacobian matrix of the reduced function to compensate for the effect of large residual, thereby enhancing the convergence performance.}
\end{itemize}

\textcolor{black}{The rest of this paper is as follows. Section \ref{sec:2} presents an analytical framework for assessing the local convergence rate of VP algorithms. In Section \ref{sec:3}, we introduce an enhanced VP algorithm for separable nonlinear problems with large residual. Section \ref{sec:4} demonstrates the efficacy of the proposed algorithm through experiments on both synthetic and real-world datasets. Finally, we summarize the main conclusions of this paper in Section \ref{sec:5}}.

\section{Convergence Analysis of Variable Projection Algorithm}
\label{sec:2}
\textcolor{black}{In this section, we introduce an analytical framework to assess how different forms of Jacobian matrix approximations influence the local convergence of the VP algorithm.} Our theoretical findings reveal that Kaufman's VP algorithm can achieve a local convergence rate comparable to that of Golub \& Pereyra's VP algorithm, thereby filling the existing research gap in this field. Additionally, we present a comparative analysis of the convergence properties of the VP algorithm and joint optimization methods.

To support the subsequent analysis, we first clarify essential symbols. Let ${\boldsymbol{a}}^*$ be a solution of minimization problem \eqref{2-2}, ${\boldsymbol{a}}$ be a feasible point, and $\{\boldsymbol{a}_k\}_{k\ge 0}$ be the sequence of iterates generated by the VP algorithm. We define the error vector as ${\boldsymbol e} = {\boldsymbol{a}} - {\boldsymbol{a}}^*$ and the error at the $n$-th iteration as ${\boldsymbol e}_n = {\boldsymbol{a}}_n - \boldsymbol{a}^*$. Furthermore, {let} $\mathcal{B}({\color{black} \boldsymbol{a}^*},{\color{black} \rho})$ represents the open ball of radius ${\color{black}\rho}$ centered at ${\boldsymbol{a}}^*$, {defined by }
\[\mathcal{B}({\color{black} \boldsymbol{a}^*},{\color{black} \rho})=\{\boldsymbol{a}|\|{\boldsymbol{a}} - {\boldsymbol{a}}^* \|<{\color{black} \rho}\}.\] Throughout this paper, we use $\nabla f(\boldsymbol{a})$ to denote the gradient vector of the objective function $f$, and $\mathbf J$ to denote the Jacobian matrix of the residual vector $\boldsymbol r$.

\subsection{{Different forms of VP algorithms}}
The following two lemmas are fundamental and important for analyzing the local convergence of different forms of VP algorithms.
\begin{lemma}
Suppose that the Jacobian matrix $\mathbf{J}(\boldsymbol{a})$ is Lipschitz continuous on the bounded set $\mathcal{B}({\color{black} \boldsymbol{a}^*},{\color{black} \rho})$, i.e., there exists a constant $L>0$ such that  $\left\|\mathbf{J}(\boldsymbol{a}_1) - \mathbf{J}(\boldsymbol{a}_2)\right\| \le L\left\|\boldsymbol{a}_1 - \boldsymbol{a}_2\right\|$ for any $\boldsymbol{a}_1, \boldsymbol{a}_2 \in \mathcal{B}({\color{black} \boldsymbol{a}^*},{\color{black} \rho})$. Then the corresponding approximate Hessian matrix $\mathbf{H}(\boldsymbol{a}) = \mathbf{J}(\boldsymbol{a})^\top\mathbf{J}(\boldsymbol{a})$ is also Lipschitz continuous on the bounded set $\mathcal{B}({\color{black} \boldsymbol{a}^*},{\color{black} \rho})$.
\label{lemma2.1}
\end{lemma}\begin{proof}
Consistent with standard assumptions in \cite{kelley1995iterative}, for \(\forall \boldsymbol{a}_1, \boldsymbol{a}_2 \in \mathcal{B}({\color{black} \boldsymbol{a}^*},{\color{black} \rho}) \), Lipschitz continuity ensures that \( \|\mathbf{J}(\boldsymbol{a}_1) - \mathbf{J}(\boldsymbol{a}_2)\| \leq L \|\boldsymbol{a}_1 - \boldsymbol{a}_2\| \), then
\begin{scriptsize} \[
\begin{aligned} 
&\| \mathbf{H}(\boldsymbol{a}_{1}) - \mathbf{H}(\boldsymbol{a}_{2}) \| = \| \mathbf{J}(\boldsymbol{a}_{1})^\top\mathbf{J}(\boldsymbol{a}_{1}) - \mathbf{J}(\boldsymbol{a}_{2})^\top\mathbf{J}(\boldsymbol{a}_{2}) \| \\
=& \| \mathbf{J}(\boldsymbol{a}_{1})^\top\mathbf{J}(\boldsymbol{a}_{1}) - \mathbf{J}(\boldsymbol{a}_{1})^\top\mathbf{J}(\boldsymbol{a}_{2}) + \mathbf{J}(\boldsymbol{a}_{1})^\top\mathbf{J}(\boldsymbol{a}_{2}) - \mathbf{J}(\boldsymbol{a}_{2})^\top\mathbf{J}(\boldsymbol{a}_{2}) \| \\
\leq& \| \mathbf{J}(\boldsymbol{a}_{1})^\top \| \| \mathbf{J}(\boldsymbol{a}_{1}) - \mathbf{J}(\boldsymbol{a}_{2}) \| + \| \mathbf{J}(\boldsymbol{a}_{2}) \| \| \mathbf{J}(\boldsymbol{a}_{1}) - \mathbf{J}(\boldsymbol{a}_{2}) \|.
\end{aligned}
\]\end{scriptsize}
Since the set \(\mathcal{B}({\color{black} \boldsymbol{a}^*},{\color{black} \rho})\) is bounded and  \(\mathbf{J}(\boldsymbol{a})\) is \(L\)-continuous, we can define \(K \overset{\mathrm{def}}{=} \sup\{\|\mathbf{J}(\boldsymbol{a})\|, \boldsymbol{a} \in \overline{\mathcal{B}({\color{black} \boldsymbol{a}^*},{\color{black} \rho})}\} < \infty\). Hence, it follows that \(\| \mathbf{H}(\boldsymbol{a}_{1}) - \mathbf{H}(\boldsymbol{a}_{2}) \| \leq 2 K L\|\boldsymbol{a}_1 - \boldsymbol{a}_2\|\), i.e., $\mathbf{H}(\boldsymbol{a})$ is Lipschitz continuous with Lipschitz constant $2KL$.
\end{proof}
\begin{lemma}
Assume that \textcolor{black}{$\mathbf{H(\boldsymbol{a}^*)}$ is non-singular} and $\mathbf{J}(\boldsymbol{a})$ is Lipschitz continuous on the bounded set $\mathcal{B}({\color{black} \boldsymbol{a}^*},{\color{black} \rho})$,  there exists ${\color{black} 0} < \sigma\leq{\color{black} \rho} $ such that for all $\boldsymbol{a} \in {\color{black}\mathcal{B}}({\color{black}\boldsymbol{a}^*},\sigma)$, the following inequality holds:
\begin{equation}
\begin{aligned} 
\|\mathbf{H}(\boldsymbol{a})\| \leq 2\|\mathbf{H}(\boldsymbol{a}^*)\|,
\label{eq:2.0a}
\end{aligned}
\end{equation}
\begin{equation}
\begin{aligned} 
\|(\mathbf{H}(\boldsymbol{a}))^{-1}\| \leq 2\|(\mathbf{H}(\boldsymbol{a}^*))^{-1}\|,
\label{eq:2.0b}
\end{aligned}
\end{equation}
\begin{equation}
\begin{aligned} 
\|\nabla f(\boldsymbol{a})\| \leq 4\|\mathbf{H}(\boldsymbol{a}^*)\|\|e\|.
\label{eq:2.0c}
\end{aligned}
\end{equation}
\label{lemma2.2}
\end{lemma}
\begin{proof}
By Lemma \ref{lemma2.1}, we have that $\mathbf{H}(\boldsymbol{a})$ is Lipschitz continuous, that is, there exists a $\gamma>0$ such that:
\begin{equation} 
\begin{split}
&\|\mathbf{H}(\boldsymbol{a})\| - \|\mathbf{H}(\boldsymbol{a}^*)\| \leq \|\mathbf{H}(\boldsymbol{a}) - \mathbf{H}(\boldsymbol{a}^*)\|  \\ \leq  & \gamma \|\boldsymbol{a} - \boldsymbol{a}^*\| = \gamma \|\boldsymbol{e}\|.
    \label{eq:2.1}
\end{split}
\end{equation}
From \eqref{eq:2.1}, by choosing an appropriate \(\sigma\) such that \( \gamma \|\boldsymbol{e}\| \leq \|\mathbf{H}(\boldsymbol{a}^*)\| \), we obtain that \( \|\mathbf{H}(\boldsymbol{a})\| \leq 2\|\mathbf{H}(\boldsymbol{a}^*)\| \) for any $\boldsymbol{a}\in \mathcal{B}({\color{black} \boldsymbol{a}^*},\sigma)$.

To verify Equation \eqref{eq:2.0b}, we consider
\[
\begin{split}
&\|\mathbf{I} - \mathbf{H}(\boldsymbol{a}^*)^{-1}\mathbf{H}(\boldsymbol{a})\| \\=& \|\mathbf{H}(\boldsymbol{a}^*)^{-1}(\mathbf{H}(\boldsymbol{a}^*) - \mathbf{H}(\boldsymbol{a}))\| \\\leq& \gamma \|\mathbf{H}(\boldsymbol{a}^*)^{-1}\| \|\boldsymbol{e}\| \\\leq& \gamma\sigma\|\mathbf{H}(\boldsymbol{a}^*)^{-1}\|.
\end{split}
\]
Hence, by choosing \( {\color{black} \sigma =\min( \frac{\|\mathbf{H}(\boldsymbol{a}^*)^{-1}\|^{-1}}{2\gamma} ,\rho)}\), we deduce that:
\[
\|\mathbf{I} - \mathbf{H}(\boldsymbol{a}^*)^{-1}\mathbf{H}(\boldsymbol{a})\| < \frac{1}{2}.
\]
Based on this, we apply Banach’s Lemma \cite{kelley1995iterative}, which implies that: 
\[
\|(\mathbf{H}(\boldsymbol{a}))^{-1}\| \leq 2\|(\mathbf{H}(\boldsymbol{a}^*))^{-1}\|.
\]
Since \( \mathbf{J}(\boldsymbol{a}) \) is Lipschitz continuous,  by a similar argument as above, we can find a ${\color{black} 0} < \sigma\leq{\color{black} \rho} $, within \( \mathcal{B}({\color{black} \boldsymbol{a}^*},\sigma) \), such that $\|\mathbf{J}(\boldsymbol{a})^\top\| \leq 2\|\mathbf{J}(\boldsymbol{a}^*)^\top\|$, $\|\mathbf{J}(\boldsymbol{a})\| \leq 2\|\mathbf{J}(\boldsymbol{a}^*)\|,
$ and the gradient \( \nabla f(\boldsymbol{a}) = \mathbf{J}(\boldsymbol{a})^\top  \boldsymbol{r}(\boldsymbol{a}) = \mathbf{J}(\boldsymbol{a})^\top \int_{0}^{1} \mathbf{J}(\boldsymbol{a}^* + t\boldsymbol{e})\boldsymbol{e}\, dt \) ,  then we have
\[\begin{aligned}
\|\nabla f(\boldsymbol{a})\|&\leq\|\mathbf{J}(\boldsymbol{a})^\top\|\int_{0}^{1}\|\mathbf{J}(\boldsymbol{a}^*+t\boldsymbol{e})\|\|\boldsymbol{e}\|dt  \\
&\leq2\|\mathbf{J}(\boldsymbol{a}^*)^\top\|\int_{0}^{1}2\|\mathbf{J}(\boldsymbol{a}^*)\|\|\boldsymbol{e}\|dt \\
&=4\|\mathbf{J}(\boldsymbol{a}^*)^\top \mathbf{J}(\boldsymbol{a}^*)\|\|\boldsymbol{e}\| \\
&=4\|\mathbf{H}(\boldsymbol{a}^*)\|\|\boldsymbol{e}\|.
\end{aligned}\]
\end{proof}
We now proceed with a local convergence analysis of the VP algorithm {applied} to  the reduced objective function \eqref{2-2}.
\begin{theorem}
   Assume that \( \boldsymbol{a}^* \) is an optimal solution of the reduced function \( {\color{black} R_2}(\boldsymbol{a}) \), and the Jacobian matrix of the residual vector $\boldsymbol{r}_2(\boldsymbol{a})$ is Lipschitz continuous and has full column rank at point \( \boldsymbol{a}^* \). Then there exist \( K > 0 \) and \( {\color{black} 0} < \sigma\leq{\color{black} \rho}  \) such that for \( \boldsymbol{a} \in \mathcal{B}({\color{black} \boldsymbol{a}^*},\sigma) \), the iterates  of the Golub $\&$ Pereyra's VP algorithm satisfy \( \|\boldsymbol{e}_{i+1}\| \le K(\|\boldsymbol{e}_{i}\|^2 + \|\nabla^2 {\color{black} R_2}(\boldsymbol{a}^*) - \mathbf{J}^\top_{\text{GP}}(\boldsymbol{a}^*) \mathbf{J}_{\text{GP}}(\boldsymbol{a}^*)\| \|\boldsymbol{e}_{i}\|) \).
    \label{theorem2.1}
\end{theorem}
\begin{proof}
    Consider the reduced SNLLS problem \eqref{2-2}, Golub \& Pereyra \cite{golub2003separable} derived the analytical expression for the Jacobian matrix {as follows}:
\begin{equation} 
\begin{split}\mathbf{J}_{\text{GP}}(\boldsymbol{a}) = -\mathbf{P}_{\boldsymbol{\Phi}}^{\perp}D\boldsymbol{\Phi}\boldsymbol{\Phi}^{\dagger}\boldsymbol{y} - (\mathbf{P}_{\boldsymbol{\Phi}}^{\perp}D\boldsymbol{\Phi}\boldsymbol{\Phi}^{\dagger})^\top\boldsymbol{y}. \label{eq:2.2}
\end{split}
\end{equation}
By the assumption that \( \mathbf{J}_{\text{GP}}(\boldsymbol{a}) \) is {of} full column rank at point \( \boldsymbol{a}^* \), we have \(\mathbf{J}^\top_{\text{GP}}(\boldsymbol{a}^*) \mathbf{J}_{\text{GP}}(\boldsymbol{a}^*) \) is nonsingular. According to Lemma \ref{lemma2.2}, there exists a \( {\color{black} 0} < \sigma\leq{\color{black} \rho} \) such that the matrix \( \mathbf{H}_{\text{GP}}(\boldsymbol{a}) = \mathbf{J}^\top_{\text{GP}}(\boldsymbol{a}) \mathbf{J}_{\text{GP}}(\boldsymbol{a}) \) is nonsingular for all {\(\boldsymbol{a}\in \mathcal{B}({\color{black} \boldsymbol{a}^*},\sigma) \)}. The iteration formula of the Golub \& Pereyra's VP algorithm is given by 
\[ \boldsymbol{a}_{i+1} = \boldsymbol{a}_{i} - (\mathbf{H}_{\text{GP}}(\boldsymbol{a}_{i}))^{-1} \mathbf{J}^\top_{\text{GP}}(\boldsymbol{a}_{i})  \boldsymbol{r}_{2}(\boldsymbol{a}_{i}), \]
from which we obtain:
\begin{equation} 
\begin{split}
&\boldsymbol{e}_{i+1} =\boldsymbol{e}_{i}-(\mathbf{H}_{\text{GP}}(\boldsymbol{a}_{i}))^{-1}\mathbf{J}^\top_{\text{GP}}(\boldsymbol{a}_{i}) \boldsymbol{r}_2(\boldsymbol{a}_{i})  \\
=&(\mathbf{H}_{\text{GP}}(\boldsymbol{a}_{i}))^{-1}\mathbf{J}^\top_{\text{GP}}(\boldsymbol{a}_{i})(\mathbf{J}_{\text{GP}}(\boldsymbol{a}_{i})\boldsymbol{e}_{i}- \boldsymbol{r}_2(\boldsymbol{a}_{i})),
\label{eq:2.3}
\end{split}
\end{equation}
where
\[
\begin{split}
&\mathbf{J}_{\text{GP}}(\boldsymbol{a}_{i})\boldsymbol{e}_{i}- \boldsymbol{r}_2(\boldsymbol{a}_{i})\\=&\mathbf{J}_{\text{GP}}(\boldsymbol{a}_{i})\boldsymbol{e}_{i}- \boldsymbol{r}_2(\boldsymbol{a}^*)+ \boldsymbol{r}_2(\boldsymbol{a}^*)- \boldsymbol{r}_2(\boldsymbol{a}_{i})\\=&- \boldsymbol{r}_2(\boldsymbol{a}^*)+(\mathbf{J}_{\text{GP}}(\boldsymbol{a}_{i})\boldsymbol{e}_{i}-( \boldsymbol{r}_2(\boldsymbol{a}_{i})- \boldsymbol{r}_2(\boldsymbol{a}^*))).
\end{split}
\]
Using the relation \(  \boldsymbol{r}_{2}(\boldsymbol{a}_{i}) -  \boldsymbol{r}_{2}(\boldsymbol{a}^*) = \mathbf{J}_{\text{GP}}(\boldsymbol{a}^*)\boldsymbol{e}_{i} + \mathcal{O}(\|\boldsymbol{e}_{i}\|^2) \), we have that 
\begin{equation} 
\begin{split}
    &\|\mathbf{J}_{\text{GP}}(\boldsymbol{a}_{i})\boldsymbol{e}_{i} - ( \boldsymbol{r}_2(\boldsymbol{a}_{i}) -  \boldsymbol{r}_2(\boldsymbol{a}^*))\| \\\leq& \|\mathbf{J}_{\text{GP}}(\boldsymbol{a}_{i}) - \mathbf{J}_{\text{GP}}(\boldsymbol{a}_{*})\| \|\boldsymbol{e}_{i}\| + \mathcal{O}(\|\boldsymbol{e}_{i}\|^2) \\\leq& L\|\boldsymbol{e}_{i}\|^2. \label{eq:2.4}
\end{split}\end{equation}
Expanding \( \mathbf{J}_{\text{GP}} \) in a Taylor series around $\boldsymbol{a}^*$, we get 
\[ \mathbf{J}_{\text{GP}}(\boldsymbol{a}_{i}) = \mathbf{J}_{\text{GP}}(\boldsymbol{a}^{*}) + \mathbf{J}_{\text{GP}}^{\prime}(\boldsymbol{a}^*)\boldsymbol{e}_{i} + \mathcal{O}(\|\boldsymbol{e}_{i}\|^2). \]
Since \( \boldsymbol{a}^* \) is a stationary point, we have that \[ \mathbf{J}_{\text{GP}}(\boldsymbol{a}^{*})^\top  \boldsymbol{r}_2(\boldsymbol{a}^*) = 0, \] and thus 
\begin{equation} 
\begin{split}
&\mathbf{J}^\top_{\text{GP}}(\boldsymbol{a}_{i})  \boldsymbol{r}_{2}(\boldsymbol{a}^*)\\ =& \boldsymbol{e}_{i}^\top \mathbf{J}_{\text{GP}}^{\prime}(\boldsymbol{a}^*) \boldsymbol{r}_2(\boldsymbol{a}^{*}) + \mathcal{O}(\|\boldsymbol{e}_{i}\|^2) \\=& \boldsymbol{e}_{i}^\top(\nabla^2 {\color{black} R_2}(\boldsymbol{a}^*)-\mathbf{H}_{\text{GP}}(\boldsymbol{a}^*)) + \mathcal{O}(\|\boldsymbol{e}_{i}\|^2). \label{eq:2.5}\end{split}\end{equation}
 Substituting \eqref{eq:2.4} and \eqref{eq:2.5} into \eqref{eq:2.3}, we \textcolor{black}{obtain} 
 \[ \begin{split}
\|\boldsymbol{e}_{i+1}\| &\le L\|( \mathbf{H}_{\text{GP}}(\boldsymbol{a}_{i}))^{-1}\| \|\mathbf{J}^\top_{\text{GP}}(\boldsymbol{a}_{i})\|\|\boldsymbol{e}_{i}\|^2 \\&+\|( \mathbf{H}_{\text{GP}}(\boldsymbol{a}_{i}))^{-1}\| \|\nabla^2 {\color{black} R_2}(\boldsymbol{a}^*)-\mathbf{H}_{\text{GP}}(\boldsymbol{a}^*)\|\|\boldsymbol{e}_{i}\| \\&+ \mathcal{O}(\|\boldsymbol{e}_{i}\|^2). \end{split}\]
By choosing an appropriate {constant} \[ K = \sup_{\boldsymbol{a}_{i} \in \overline{\mathcal{B}({\color{black} \boldsymbol{a}^*},\sigma)}} \|( \mathbf{H}_{\text{GP}}(\boldsymbol{a}_{i}))^{-1}\| (1+L\|\mathbf{J}^\top_{\text{GP}}(\boldsymbol{a}_{i})\|), \] we conclude that \[ \|\boldsymbol{e}_{i+1}\| \le K(\|\boldsymbol{e}_{i}\|^2 + \|\nabla^2{\color{black} R_2}(\boldsymbol{a}^*)-\mathbf{H}_{\text{GP}}(\boldsymbol{a}^*)\|\|\boldsymbol{e}_{i}\|).\]
\end{proof}
\textbf{\it {Remark}.1}: 
    For problems with low nonlinearity, such as low-rank matrix decomposition problems, the difference between \(\nabla^2 {\color{black} R_2}(\boldsymbol{a})\) and \(\mathbf{H}_{\text{GP}}(\boldsymbol{a})\) is often negligible. In these cases, the Golub \& Pereyra's VP algorithm can achieve superlinear convergence rates, and even quadratic convergence rates under certain conditions. 

    Next, we shift our focus to the local convergence rate of Kaufman’s VP algorithm. The Jacobian matrix proposed by Kaufman \cite{kaufman1975variable} simplifies the expression of \eqref{eq:2.2} by {omitting} the second term: 
    \begin{equation}    
     \mathbf{J}_{\text{Kau}}(\boldsymbol{a}) = -\mathbf{P}_{\boldsymbol{\Phi}}^{\perp}D\boldsymbol{\Phi}\boldsymbol{\Phi}^{\dagger}\boldsymbol{y}.
     \label{eq:2.5.5}
    \end{equation}
Letting \( \mathbf{H}_{\text{Kau}}(\boldsymbol{a}) = \mathbf{J}^\top_{\text{Kau}}(\boldsymbol{a}) \mathbf{J}_{\text{Kau}}(\boldsymbol{a}) \), {we can express Kaufman's Jacobian matrix as} \( \mathbf{J}_{\text{Kau}}(\boldsymbol{a}) = \mathbf{J}_{\text{GP}}(\boldsymbol{a}) + \delta(\boldsymbol{a}) \), where \( \delta(\boldsymbol{a}) = (\mathbf{P}_{\boldsymbol{\Phi}}^{\perp}D\boldsymbol{\Phi}\boldsymbol{\Phi}^{\dagger})^\top\boldsymbol{y} \). Consequently, the corresponding approximated Hessian matrix can be articulated as:
\[\begin{split}
 \mathbf{H}_{\text{Kau}}(\boldsymbol{a}) &= (\mathbf{J}_{\text{GP}}(\boldsymbol{a}) + \delta(\boldsymbol{a}))^\top (\mathbf{J}_{\text{GP}}(\boldsymbol{a}) + \delta(\boldsymbol{a})) \\&= \mathbf{J}^\top_{\text{GP}}(\boldsymbol{a})\mathbf{J}_{\text{GP}}(\boldsymbol{a}) + \delta(\boldsymbol{a})^\top \mathbf{J}_{\text{GP}}(\boldsymbol{a}) \\&\ \ \ +\mathbf{J}^\top_{\text{GP}}(\boldsymbol{a}) \delta(\boldsymbol{a}) + \delta(\boldsymbol{a})^\top \delta(\boldsymbol{a}) \\&\ \dot{=} \mathbf{H}_{\text{GP}}(\boldsymbol{a}) + \mathbf\Delta(\boldsymbol{a}). \end{split}\]
\begin{theorem}
Let \({\color{black} R_2}(\boldsymbol{a}) \) denote a reduced function that is second-order Lipschitz continuously differentiable. Assume that \( \boldsymbol{a}^* \) is a critical point and that the Hessian matrix \(\nabla^2 {\color{black} R_2}(\boldsymbol{a}^*) \) is positive definite. Within a bounded open neighborhood of \( \boldsymbol{a}^* \), the Jacobian matrices \( \mathbf{J}_{\text{GP}} \) and \( \mathbf{J}_{\text{Kau}} \) are Lipschitz continuous, and  \(\mathbf\Delta(\boldsymbol{a}) \) satisfies \( \| \mathbf\Delta(\boldsymbol{a})\| \leq \|\mathbf{H}^{-1}_{\text{GP}}(\boldsymbol{a}^*)\|^{-1}/4 \). Then, there exist positive {constants} \( K \), \( \tau_1 \), and \( \tau_2 \), such that the error after updating the nonlinear parameters by Kaufman's VP algorithm satisfies the inequality: \[\begin{split}
\|\boldsymbol{e}_{i+1}^{\text{Kau}}\| &\le K(\|\boldsymbol{e}_{i}\|^2 + \|\nabla^2 {\color{black} R_2}(\boldsymbol{a}^*) - \mathbf{H}_{\text{GP}}(\boldsymbol{a}^*)\| \|\boldsymbol{e}_{i}\|) \\& \ \ \ + 16K \|\mathbf{H}_{\text{GP}}^{-1}(\boldsymbol{a}^*)\| \|\nabla^2 {\color{black} R_2}(\boldsymbol{a}^*)\| \|\mathbf\Delta(\boldsymbol{a}_{i})\| \|\boldsymbol{e}_{i}\| \\&\dot{=} K(\|\boldsymbol{e}_{i}\|^2 + \tau_1 \|\boldsymbol{e}_{i}\| + \tau_2 \|\mathbf\Delta(\boldsymbol{a}_{i})\| \|\boldsymbol{e}_{i}\|).\end{split}\]
\label{theorem2.2}
\end{theorem}
\begin{proof}
 The update equations for the nonlinear parameters obtained by the VP algorithm in the forms of Golub \& Pereyra and Kaufman are:
\begin{scriptsize}\[\begin{split}
&\boldsymbol{a}_{i+1}^{\text{GP}} = \boldsymbol{a}_{i} - \mathbf{H}^{-1}_{\text{GP}}(\boldsymbol{a}_{i}) \mathbf{J}^\top_{\text{GP}}(\boldsymbol{a}_{i}) \boldsymbol{r}_2(\boldsymbol{a}_{i}),\\&\boldsymbol{a}_{i+1}^{\text{Kau}} = \boldsymbol{a}_{i} - \mathbf{H}^{-1}_{\text{Kau}}(\boldsymbol{a}_{i}) \mathbf{J}^\top_{\text{Kau}}(\boldsymbol{a}_{i}) \boldsymbol{r}_2(\boldsymbol{a}_{i}) \\&= \boldsymbol{a}_{i+1}^{\text{GP}} + (\mathbf{H}^{-1}_{\text{GP}}(\boldsymbol{a}_{i}) - \mathbf{H}^{-1}_{\text{Kau}}(\boldsymbol{a}_{i})) \mathbf{J}^\top_{\text{GP}}(\boldsymbol{a}_{i}) \boldsymbol{r}_2(\boldsymbol{a}_{i}) \\&= \boldsymbol{a}_{i+1}^{\text{GP}} + (\mathbf{H}^{-1}_{\text{GP}}(\boldsymbol{a}_{i}) - (\mathbf{H}_{\text{GP}}(\boldsymbol{a}_{i}) + \mathbf\Delta(\boldsymbol{a}_{i}))^{-1}) \mathbf{J}^\top_{\text{GP}}(\boldsymbol{a}_{i}) \boldsymbol{r}_2(\boldsymbol{a}_{i}).
\end{split}
\]\end{scriptsize}
According to \cite{Gan2018On}, the gradient can be determined by the product of different Jacobian matrices and residual vectors, i.e., \(\mathbf{J}^\top_{\text{GP}}(\boldsymbol{a}) \boldsymbol{r}_2(\boldsymbol{a})=\mathbf{J}^\top_{\text{Kau}}(\boldsymbol{a}) \boldsymbol{r}_2(\boldsymbol{a})\), therefore, 
\begin{scriptsize}
\begin{equation} 
    \|\boldsymbol{e}_{i+1}^{\text{Kau}}\| \le \|\boldsymbol{e}_{i+1}^{\text{GP}}\| + \|\mathbf{H}^{-1}_{\text{GP}}(\boldsymbol{a}_{i}) - (\mathbf{H}_{\text{GP}}(\boldsymbol{a}_{i}) + \mathbf\Delta(\boldsymbol{a}_{i}))^{-1}\| \|\nabla f(\boldsymbol{a}_{i})\label{eq:2.6}\|. 
\end{equation}\end{scriptsize}
By the assumptions of \( {\color{black} R_2}(\boldsymbol{a}) \), (i.e., the reduced function \( {\color{black} R_2} \) is second-order Lipschitz continuously differentiable, {with} \( \nabla {\color{black} R_2}(\boldsymbol{a}^*) = 0 \)  and  a positive definite Hessian \( \nabla^2 {\color{black} R_2}(\boldsymbol{a}^*) \) ), it can be discerned that \( \|\nabla {\color{black} R_2}(\boldsymbol{a}_{i})\| \) is bounded, and \begin{equation}
    \|\mathbf{J}^\top_{\text{GP}}(\boldsymbol{a}_{i}) \boldsymbol{r}_2(\boldsymbol{a}_{i})\| \leq 2\|\nabla^2 {\color{black} R_2}(\boldsymbol{a}^*)\| \|\boldsymbol{e}_{i}\|.\label{eq:2.7}\end{equation}
Additionally, since \( \mathbf{J}_{\text{GP}}(\boldsymbol{a}) \) is \( L \)-continuous, and combining Lemma \ref{lemma2.2} and the conclusions drawn in \cite{kelley1995iterative}, it can be deduced  that there exists \( {\color{black} 0} < \sigma\leq{\color{black} \rho} \) such that for all {\( \boldsymbol{a} \in  \mathcal{B}({\color{black} \boldsymbol{a}^*},\sigma) \)}, \[ \begin{split}
\|\mathbf{H}_{\text{GP}}(\boldsymbol{a})\| \leq 2\|\mathbf{H}_{\text{GP}}(\boldsymbol{a}^*)\|, \\ \|\mathbf{H}_{\text{GP}}^{-1}(\boldsymbol{a})\| \leq 2\|\mathbf{H}_{\text{GP}}^{-1}(\boldsymbol{a}^*)\|. \end{split}\] Furthermore, given that \( \|\mathbf\Delta(\boldsymbol{a})\| \leq \|\mathbf{H}^{-1}_{\text{GP}}(\boldsymbol{a}^*)\|^{-1}/4 \), it follows that \[ \|\mathbf\Delta(\boldsymbol{a})\| \leq \|\mathbf{H}^{-1}_{\text{GP}}(\boldsymbol{a}_{i})\|^{-1}/2. \]
According to the Banach Lemma \cite{kelley1995iterative}, the matrix \( \mathbf{H}_{\text{GP}}(\boldsymbol{a}_{i}) + \mathbf\Delta(\boldsymbol{a}_{i}) \) is nonsingular, and {we have} \[ \|(\mathbf{H}_{\text{GP}}(\boldsymbol{a}_{i}) + \mathbf\Delta(\boldsymbol{a}_{i}))^{-1}\| \leq 2\|\mathbf{H}^{-1}_{\text{GP}}(\boldsymbol{a}_{i})\| \leq 4\|\mathbf{H}_{\text{GP}}^{-1}(\boldsymbol{a}^*)\|. \]
Using the Banach Lemma \cite{kelley1995iterative} again, we arrive at the bound
\begin{scriptsize}\begin{equation} \|\mathbf{H}^{-1}_{\text{GP}}(\boldsymbol{a}_{i}) - (\mathbf{H}_{\text{GP}}(\boldsymbol{a}_{i}) + \mathbf\Delta(\boldsymbol{a}_{i}))^{-1}\| \leq 8\|\mathbf{H}^{-1}_{\text{GP}}(\boldsymbol{a}^*)\|^2\|\mathbf\Delta(\boldsymbol{a}_{i})\|. \label{eq:2.8}\end{equation}\end{scriptsize}
Substituting \eqref{eq:2.7} and \eqref{eq:2.8} into equation \eqref{eq:2.6}, we obtain \[ \begin{split}
\|\boldsymbol{e}^{\text{Kau}}_{i+1}\| \leq &K(\|\boldsymbol{e}_{i}\|^2 + \|\nabla^2 {\color{black} R_2}(\boldsymbol{a}^*) - \mathbf{H}_{\text{GP}}(\boldsymbol{a}^*)\| \|\boldsymbol{e}_{i}\|) \\&+ 16\|\mathbf{H}^{-1}_{\text{GP}}(\boldsymbol{a}^*)\| \|\nabla^2 {\color{black} R_2}(\boldsymbol{a}^*)\| \|\mathbf\Delta(\boldsymbol{a}_{i})\| \|\boldsymbol{e}_{i}\|. \end{split}\]
Denote \( \tau_1 = \|\nabla^2 {\color{black} R_2}(\boldsymbol{a}^*) - \mathbf{H}_{\text{GP}}(\boldsymbol{a}^*)\| \) and \( \tau_2 = 16 \|\mathbf{H}^{-1}_{\text{GP}}(\boldsymbol{a}^*)\| \|\nabla^2 {\color{black} R_2}(\boldsymbol{a}^*)\| \), then \[ \|\boldsymbol{e}^{\text{Kau}}_{i+1}\| \leq K(\|\boldsymbol{e}_{i}\|^2 + \tau_1 \|\boldsymbol{e}_{i}\|) + \tau_2 \|\mathbf\Delta(\boldsymbol{a}_{i})\| \|\boldsymbol{e}_{i}\|.\]
\end{proof}
\textcolor{black}{Theorem \ref{theorem2.2} shows the local convergence behavior of Kaufman's VP algorithm}. Comparing {Theorems} \ref{theorem2.1} and \ref{theorem2.2}, it can be found that the Golub \& Pereyra's VP algorithm and Kaufman's VP algorithm can achieve similar {local convergence rate}. To facilitate comparison, the local convergence rates of both {algorithms} are listed as follows: \[ \begin{split}
 \|\boldsymbol{e}^{\text{GP}}_{i+1}\| &\leq K(\|\boldsymbol{e}_{i}\|^2 + \tau_1 \|\boldsymbol{e}_{i}\|),\\  \|\boldsymbol{e}^{\text{Kau}}_{i+1}\| &\leq K(\|\boldsymbol{e}_{i}\|^2 + \tau_1 \|\boldsymbol{e}_{i}\|) + \tau_2 \|\mathbf\Delta(\boldsymbol{a}_{i})\| \|\boldsymbol{e}_{i}\|. \end{split}\]Intuitively, when the residual \(  \boldsymbol{r}_2(\boldsymbol{a}_i) = \mathbf{P}_{\boldsymbol{\Phi}}^{\perp}\boldsymbol{y} \) is small, \( \|\mathbf\Delta(\boldsymbol{a}_i)\| = \|(\mathbf{P}_{\boldsymbol{\Phi}}^{\perp}D\boldsymbol{\Phi}\boldsymbol{\Phi}^{\dagger})^\top\boldsymbol{y}\| = \|(D\boldsymbol{\Phi}\boldsymbol{\Phi}^{\dagger})^\top\mathbf{P}_{\boldsymbol{\Phi}}^{\perp}\boldsymbol{y}\| \) is also relatively small. In such case,  the  Golub \& Pereyra's and Kaufman's VP {achieve} similar local convergence rates. A similar analysis can be presented from another perspective. Noting that \[ \begin{split}
\mathbf{J}_{\text{GP}}(\boldsymbol{a})&=\frac{\partial \boldsymbol{r}_2}{\partial \boldsymbol{a}}=\frac{\partial \boldsymbol{r}_2}{\partial \boldsymbol{a}}+\frac{\partial \boldsymbol{r}_2}{\partial \boldsymbol{c}}\frac{d \boldsymbol{c}}{d \boldsymbol{a}}\\&=-D\boldsymbol{\Phi}\boldsymbol{\Phi}^{\dagger}\boldsymbol{y}-\boldsymbol{\Phi}\frac{d \boldsymbol{c}}{d \boldsymbol{a}}.\end{split}\]
{Chen et al. \cite{chen2021insights} } pointed out that the Kaufman's VP algorithm essentially takes a first-order linear approximation when calculating the derivatives of linear parameters to nonlinear parameters, that is, \(\Delta \boldsymbol{c} \) and \( \Delta \boldsymbol{a} \) satisfy \( \Delta \boldsymbol{c} = -(\boldsymbol{\Phi}^{\dagger}D\boldsymbol{\Phi}\boldsymbol{\Phi}^{\dagger}\boldsymbol{y}) \Delta \boldsymbol{a}\). The derivative \( \frac{d \boldsymbol{c}}{d \boldsymbol{a}} \) can be expressed as \( \frac{d \boldsymbol{c}}{d \boldsymbol{a}} = \frac{\Delta \boldsymbol{c}}{\Delta \boldsymbol{a}} + {\color{black}\zeta} = -\boldsymbol{\Phi}^{\dagger}D\boldsymbol{\Phi}\boldsymbol{\Phi}^{\dagger} + {\color{black}\zeta} \), {leading to the relationship for the Jacobian matrices}
\[
\begin{aligned}
\mathbf{J}_{\text{GP}}(\boldsymbol{a}) &= -D\boldsymbol{\Phi}\boldsymbol{\Phi}^{\dagger}\boldsymbol{y} - \boldsymbol{\Phi}\left[-\boldsymbol{\Phi}^{\dagger}D\boldsymbol{\Phi}\boldsymbol{\Phi}^{\dagger}\boldsymbol{y} + {\color{black}\zeta}\right] \\
&= -\mathbf{P}_{\boldsymbol{\Phi}}^{\perp}D\boldsymbol{\Phi}\boldsymbol{\Phi}^{\dagger}\boldsymbol{y} - \boldsymbol{\Phi}{\color{black}\zeta} = \mathbf{J}_{\text{Kau}}(\boldsymbol{x}) - \mathcal O({\color{black}\zeta}).
\end{aligned}
\]
When the error in the first-order linear approximation is small, the discrepancy between the two forms of the Jacobian matrices is relatively small, leading to comparable local convergence rates for Golub \& Pereyra's and Kaufman's VP algorithms.

\subsection{{Comparison of VP Algorithm and Joint Algorithm}}
In the following part, we {apply} similar analytical methods to compare the local convergence of the joint optimization algorithm with the VP algorithm. For clarity, we redefine some of the relevant symbols. Let the vector \( \boldsymbol{\theta} = [\boldsymbol{c}^\top,\boldsymbol{a}^\top]^\top \), \(  \boldsymbol{\epsilon}_i^{\boldsymbol{\theta}} = \boldsymbol{\theta}_i - \boldsymbol{\theta}^* \), \(  \boldsymbol{\epsilon}_i^{\boldsymbol{a}} = \boldsymbol{a}_i - \boldsymbol{a}^* \). 

Following the analytical approach similar to Theorem \ref{theorem2.1}, when applying the Joint optimization method, the parameter error at the \((i+1)\)-th iteration can be represented as: \begin{scriptsize}\begin{equation} \text{Joint: }
\| \boldsymbol{\epsilon}^{\boldsymbol{\theta}}_{i+1}\| \leq K_{\text{J}}(\| \boldsymbol{\epsilon}^{\boldsymbol{\theta}}_{i}\|^2 + \|\nabla^2 {\color{black} R_2}(\boldsymbol{\theta}^*) - \mathbf{J}^\top(\boldsymbol{\theta}^*) \mathbf{J}(\boldsymbol{\theta}^*)\| \| \boldsymbol{\epsilon}^{\boldsymbol{\theta}}_{i}\|),\label{eq:2.9}\end{equation}\end{scriptsize}
where \( K_{\text{J}} = \sup_{\boldsymbol{a}_{i} \in \overline{\mathcal{B}({\color{black} \boldsymbol{a}^*},{\color{black} \rho})}} \|(\mathbf{H}(\boldsymbol{\theta}_{i}))^{-1}\| (1+L\|\mathbf{J}^\top(\boldsymbol{\theta}_{i})\|) \), {with} $\mathbf{J}$ and $\mathbf{H}$ denoting the Jacobian and Hessian matrices of the original function $r(\boldsymbol{a},\boldsymbol{c})$, respectively. Furthermore, according to Theorem \ref{theorem2.1}, the update formula of the VP algorithm satisfies: \begin{scriptsize}\begin{equation}\text{VP: }
\| \boldsymbol{\epsilon}^{\boldsymbol{a}}_{i+1}\| \leq K(\| \boldsymbol{\epsilon}^{\boldsymbol{a}}_{i}\|^2 + \|\nabla^2 {\color{black} R_2}(\boldsymbol{a}^*) - \mathbf{J}_{\text{GP}}^\top(\boldsymbol{a}^*) \mathbf{J}_{\text{GP}}(\boldsymbol{a}^*)\| \| \boldsymbol{\epsilon}^{\boldsymbol{a}}_{i}\|).\label{eq:2.10}\end{equation}\end{scriptsize}
Unlike the joint optimization method, the linear parameter \( \boldsymbol{c} \) in the VP method is determined by solving the optimization problem with fixed nonlinear parameter \( \boldsymbol{a} \): 
\[ \hat{\boldsymbol{c}}(\boldsymbol{a}) = \arg\min_{\boldsymbol{c}} r(\boldsymbol{c}, \boldsymbol{a}) = \boldsymbol{\Phi}^{\dagger}(\boldsymbol{a})\boldsymbol{y}.\]
Therefore, the error of the linear parameters at the (\(i+1\))-th iteration can be bounded by the following upper bound: \[ \begin{split}
\|  \boldsymbol{\epsilon}_{i+1}^{\boldsymbol{c}} \| &= \| \hat{\boldsymbol{c}}_{i+1} - \boldsymbol{c}^* \| \\&\le \| \boldsymbol{\Phi}^{\dagger}(\boldsymbol{a}_{i+1}) - \boldsymbol{\Phi}^{\dagger}(\boldsymbol{a}^*) \| \| \boldsymbol{y} \| \\&\le (\| D\boldsymbol{\Phi}^{\dagger}(\boldsymbol{a}^*) \| \|  \boldsymbol{\epsilon}^{\boldsymbol{a}}_{i+1} \| + \mathcal{O}(\|  \boldsymbol{\epsilon}^{\boldsymbol{a}}_{i+1} \|^2)) \| \boldsymbol{y} \|. \end{split}\]
Denote \( \eta = (\| D\boldsymbol{\Phi}^{\dagger}(\boldsymbol{a}^*) \| + \mathcal{O}(\|  \boldsymbol{\epsilon}^{\boldsymbol{a}}_{i+1} \|) \| \boldsymbol{y} \| \), we subsequently can obtain  \begin{equation}
\|  \boldsymbol{\epsilon}_{i+1}^{\boldsymbol{c}} \| \le \eta \|  \boldsymbol{\epsilon}_{i+1}^{\boldsymbol{a}} \|. \label{eq:2.12}\end{equation}
{By} combining \eqref{eq:2.10} and \eqref{eq:2.12}, we can express the error of all parameters \(\boldsymbol \theta \) in the model estimated by the VP algorithm at the ($i+1$)-th iteration as:
\begin{equation}
\| \boldsymbol{\epsilon}^{\boldsymbol{\theta}}_{i+1}\| \leq K \| \boldsymbol{\epsilon}_i^{\boldsymbol{a}}\|^2 + K_1 \| \boldsymbol{\epsilon}_i^{\boldsymbol{a}}\|, \label{eq:2.13}\end{equation}
where \(K_1 = K\|\nabla^2 {\color{black} R_2}(\boldsymbol{a}^*) - \mathbf{J}_{\text{GP}}^\top(\boldsymbol{a}^*) \mathbf{J}_{\text{GP}}(\boldsymbol{a}^*)\|+\eta\). By comparing equations \eqref{eq:2.13} and \eqref{eq:2.9}, {it is evident} that the bound of the parameter error of the VP method is only determined by the error of the nonlinear parameters. This advantage in convergence speed becomes particularly evident when the linear parameters have a high dimensionality, which aligns with the findings observed in the empirical studies \cite{erichson2020sparse,espanol2023variable,aravkin2017efficient}.

\section{Variable Projection for Separable Nonlinear Problem with Large Residual}
\label{sec:3}
\textcolor{black}{Building on the convergence analysis of the VP algorithm, this section introduces a more efficient separable optimization method for nonlinear optimization problems with relatively large residual. Initially, after computing the Jacobian matrix of the reduced objective, the approximate Hessian matrix can be expressed as:}
\begin{equation*}
    \mathbf{H}^k\approx \mathbf{J}^{\top}({\color{black}\boldsymbol{a}^k})\mathbf{J}({\color{black}\boldsymbol{a}^k}),
\end{equation*}
where the Jacobian matrix could adopt different forms such as the Golub \& Pereyra's form {and} Kaufman's form. For {consistency}, we denote it as $\mathbf{J}({\color{black}\boldsymbol{a}})$. 
For problems with a small residual, the Gauss-Newton (GN) or Levenberg-Marquardt (LM) updates are generally effective in optimizing the nonlinear parameters of the objective. However, in practical scenarios, we frequently encounter problems with large residual. In such cases, the second term of the Hessian matrix of the reduced objective
\begin{equation}\label{3-1}
    \mathbf{H}^k=\mathbf{J}^{\top}({\color{black}\boldsymbol{a}^k})\mathbf{J}({\color{black}\boldsymbol{a}^k})+\mathbf T^k
\end{equation}
includes at least $\sum\limits_{j=1}^m{ \boldsymbol{r}_{2_j}}(\boldsymbol{a}^k)\nabla^2 { \boldsymbol{r}_{2_j}}(\boldsymbol{a}^k)$, is non-negligible and exerts a considerable influence on the convergence of the algorithm. This can also be inferred from our theoretical analysis (Theorem \ref{theorem2.2}), which suggests that when the residual {is} large, the deviation of the approximate Hessian matrix significantly impacts the convergence of the algorithm. This necessitates the design of a VP algorithm for Large Residual (VPLR) tailored for separable nonlinear optimization problems with large residual. The proposed VPLR algorithm addresses the {coupling relationships between different parts of parameters within the model and recursively corrects the approximate Hessian matrix, thereby improving the overall performance of the algorithm.}

More explicitly, our objective is for $\mathbf{T}^{k+1}$ to bear as much resemblance to $\sum\limits_{j=1}^m{ \boldsymbol{r}_{2_j}}(\boldsymbol{a}^{k+1})\nabla^2 { \boldsymbol{r}_{2_j}}(\boldsymbol{a}^{k+1})$ as feasible. {Based on} the first-order Taylor expansion, $\mathbf{T}^k$ should endeavor to preserve the characteristics of the original Hessian matrix to the greatest extent possible, as illustrated in:
\begin{small}
\begin{equation}
\begin{split}
    \mathbf{T}^{k+1}\boldsymbol{s}^k&\approx \left(\sum\limits_{j=1}^m{ \boldsymbol{r}_{2_j}}(\boldsymbol{a}^{k+1})\nabla^2 { \boldsymbol{r}_{2_j}}(\boldsymbol{a}^{k+1})\right)\boldsymbol{s}^k\\&=\sum\limits_{j=1}^m{ \boldsymbol{r}_{2_j}}(\boldsymbol a^{k+1})(\nabla^2{\boldsymbol{r}_{2_j}}(\boldsymbol{a}^{k+1}))\boldsymbol{s}^k\\& \approx \sum\limits_{j=1}^m{ \boldsymbol{r}_{2_j}}(\boldsymbol{a}^{k+1})(\nabla { \boldsymbol{r}_{2_j}}(\boldsymbol{a}^{k+1})-\nabla { \boldsymbol{r}_{2_j}}(\boldsymbol{a}^{k})) \\& =\mathbf{J}^\top({\color{black}\boldsymbol{a}^{k+1}}) \boldsymbol{r}_2^{k+1}-\mathbf{J}^\top({\color{black}\boldsymbol{a}^{k}}) \boldsymbol{r}_2^{k+1},\label{3-2}
\end{split}
\end{equation}
\end{small}
where $\boldsymbol{s}^k=\boldsymbol{a}^{k+1}-\boldsymbol{a}^k$. Let $\hat{\boldsymbol{g}}^k=\mathbf{J}^\top({\color{black}\boldsymbol{a}^{k+1}}) \boldsymbol{r}_2^{k+1}-\mathbf{J}^\top({\color{black}\boldsymbol{a}^{k}}) \boldsymbol{r}_2^{k+1}$, then the condition that the correction term $\mathbf{T}^k$ is required to satisfy can be articulated as follows:
\begin{equation}
\mathbf{T}^{k+1}\boldsymbol{s}^k=\hat{\boldsymbol{g}}^k.\label{3-3}
\end{equation}
\eqref{3-3} {resembles} the secant equation prevalent in Quasi-Newton methods \cite{dennis1996numerical}. Consequently, we draw inspiration from the updating mechanism of the Quasi-Newton matrix to effectuate the update of the {correction} matrix $\mathbf{T}^k$, as delineated below:
\begin{equation}
    \mathbf{T}^{k+1} = \mathbf{T}^{k} - \frac{\mathbf{T}^{k} \boldsymbol{s}_{k} \boldsymbol{s}_{k}^\top \mathbf{T}^{k}}{\boldsymbol{s}_{k}^\top \mathbf{T}^{k} \boldsymbol{s}_{k}} + \frac{\hat{\boldsymbol{g}}_{k} \hat{\boldsymbol{g}}_{k}^\top}{\hat{\boldsymbol{g}}_{k}^\top \boldsymbol{s}_{k}}.\label{3-4}
\end{equation}
It is important to distinguish \eqref{3-4} from the Broyden–Fletcher–Goldfarb–Shanno variant of Quasi-Newton methods. The term $\hat{\boldsymbol{g}}_k$ in \eqref{3-4} does not signify the gradient difference between two consecutive iterations. Once the Hessian matrix of the reduced objective function is obtained, the direction of update for the nonlinear parameters $\boldsymbol{a}$ can be {determined} by resolving the following equation: 
\begin{equation}
\left(\mathbf{J}^{\top}({\color{black}\boldsymbol{a}^k})\mathbf{J}({\color{black}\boldsymbol{a}^k})+\mathbf{T}^k\right){\boldsymbol{d}}^k_{\boldsymbol{a}} =-\mathbf{J}^{\top}({\color{black}\boldsymbol{a}^k})  \boldsymbol{r}_2(\boldsymbol{a}^k).\label{3-5}
\end{equation}
This facilitates the computation of the update for the nonlinear parameters:
\begin{equation}
    \boldsymbol{a}^{k+1}=\boldsymbol{a}^k+\beta^k {\boldsymbol{d}}^k_{\boldsymbol{a}},\label{3-6}
\end{equation}
where \( \beta^k \) denotes the step size at the \( k \)-th iteration, determined through a backtracking line search strategy that \textcolor{black}{satisfies the Armijo condition} \cite{armijo1966minimization,nocedal1999numerical} to ensure a sufficient decrease in the objective function at each iteration.

The algorithm {outlined} above, VPLR, specifically designed for separable nonlinear optimization problems with large residual, is summarized in Algorithm \ref{alg:1}. By compensating for the impact of large residual through equations \eqref{3-1}-\eqref{3-4}, and in conjunction with the fundamental conclusion of Theorem \ref{theorem2.2}, it is not difficult to discern that our proposed VPLR demonstrates a superior convergence rate compared to the traditional VP algorithm.

\begin{algorithm}[htbp]
\caption{Refined variable projection algorithm for separable nonlinear problem with large residual, VPLR}
\begin{algorithmic}
\label{alg:1}
\STATE /*\textit{ \textbf{Initialize} }*/
\STATE Choose $\boldsymbol{a}^0\in \mathbb{R}^m$ , a small positive number $\varepsilon$ and the maximum iteration $N$. Set $k = 0$.
\STATE /*\textit{ \textbf{Estimate parameters} }*/
\STATE \textbf{Step1:} Eliminate the linear parameters $\boldsymbol{c}$ by solving a linear subprogram \eqref{2-1} to obtain a reduced objective function ${\color{black} R_2}(\boldsymbol{a})$.
\STATE \textbf{Step2:} Compute the Jacobian matrix $\mathbf{J}({\color{black}\boldsymbol{a}^k})$ adhering to either \eqref{eq:2.2} or \eqref{eq:2.5.5} (manifested in the form of Golub \& Pereyra or Kaufman).
\STATE \textbf{Step3:} Resolve \eqref{3-5} to determine the update direction ${\boldsymbol{d}}^k_{\boldsymbol{a}}$, subsequently update the nonlinear parameters $\boldsymbol{a}$ according to \eqref{3-6}.
\STATE \textbf{Step4:} Proceed to iteratively update the correction matrix $\mathbf{T}^{k+1}$ employing the update rule outlined in equation \eqref{3-4}.
\STATE \textbf{Step5:} if $|{\color{black} R_2}(\boldsymbol{a}^{k+1})-{\color{black} R_2}(\boldsymbol{a}^{k})|<\varepsilon$ or $\|\boldsymbol{a}^{k+1}-\boldsymbol{a}^{k}\|<\varepsilon$ or $k>N$
\STATE $\quad$ Terminate the process;
\STATE else
\STATE $\quad$ $k=k+1$, go to Step 1;
\STATE end
\STATE \textbf{Output}
\STATE Nonlinear parameters: $\hat{\boldsymbol{a}} = \boldsymbol{a}^{k+1}$
\STATE Linear parameters: $\hat{\boldsymbol{c}} = \arg\min_{\boldsymbol{c}} r(\boldsymbol{c},\hat{\boldsymbol{a}})=\mathbf{\Phi}^\dagger\boldsymbol{y}$.
\end{algorithmic}
\end{algorithm}

\section{NUMERICAL EXAMPLES}
\label{sec:4}
In this section, we present the results of numerical experiments on both synthetic and real-world datasets to {evaluate} the effectiveness of the proposed VPLR algorithm. We perform a comprehensive evaluation and comparison of our algorithm with different optimization algorithms, including the traditional VP algorithm; the Joint Optimization algorithm, which \textcolor{black}{employs the GN method} to optimize all model parameters simultaneously; the Alternating Minimization algorithm (AM) \cite{guminov2021combination}, which alternately \textcolor{black}{updates} between optimizing linear and nonlinear parameters within the model; and Block Coordinate Descent (BCD) \cite{nutini2022let,wright2015coordinate}, \textcolor{black}{which is also implemented using the GN method}. \textcolor{black}{The termination criteria for all algorithms are as follows: (1) the relative change in the objective function value between two consecutive iterations satisfies \( |r^k - r^{k-1}| < 10^{-10} |r^k| \), where \( r^k \) represents the value of the objective function \( r(\boldsymbol{a}, \boldsymbol{c}) \) at the \( k \)-th iteration; or (2) the maximum number of iterations reaches 100. To ensure fairness, the initial step size for all algorithms is uniformly set to 1.} All experiments were carried out on a PC with a 2.30 GHz CPU and 16GB RAM, using Matlab 2022 as the programming environment.

\subsection{Parameter Estimation of Exponential Model}
\textcolor{black}{We consider a exponential model of the form}
\begin{equation} 
\begin{split}
    &f({x}) =\ c_1 e^{-a_2 {x}^2}\cos(a_3 {x})+c_2 e^{-a_1 {x}^2}\cos(a_2 {x})\\&\quad \quad \quad \ +c_3 e^{-a_4 {x}^2}\sin(a_1 {x})+\epsilon, \\& {\boldsymbol{c}}=[c_1,c_2,c_3]^\top = [2,3,2]^\top,\\& {\boldsymbol{a}}=[a_1,a_2,a_3,a_4]^\top = [10, 15, 30, 8]^\top,\label{eq:4.1}
\end{split}
\end{equation}
where $\boldsymbol{c}$ denotes the linear parameter, $\boldsymbol{a}$ represents the nonlinear parameter, {and} $\epsilon$ is the noise. Utilizing these parameter {settings}, a dataset consisting of 200 observations is generated by uniformly sampling the interval $[0, 1]$ with a step size of 0.005. The resulting data are visualized in Fig. \ref{fig:4.1}.
\begin{figure}[htbp]
    \centering
        {\includegraphics[width=0.6\linewidth]{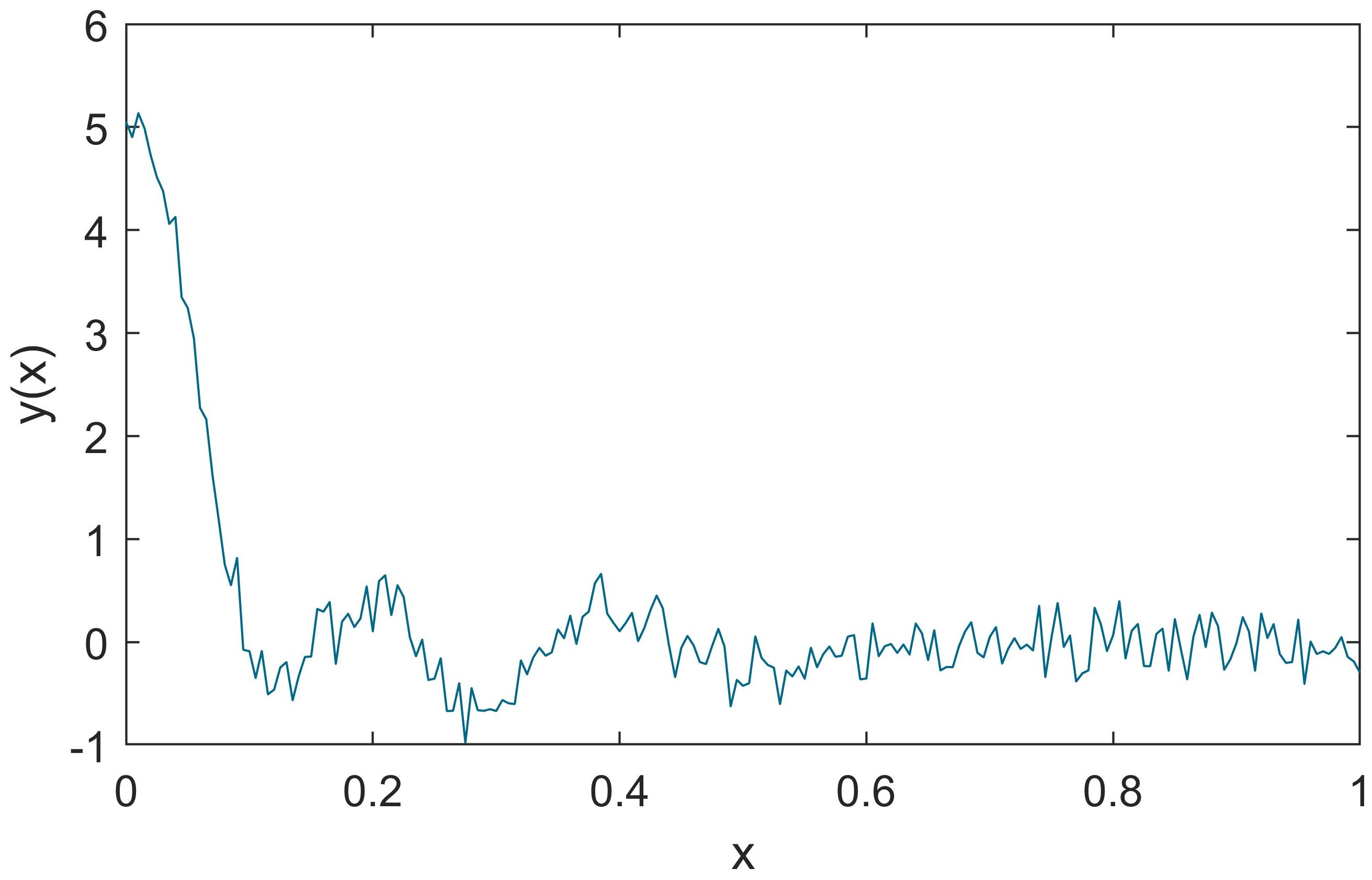}}
    \caption{Observational data generated from the exponential model}
    \label{fig:4.1}
\end{figure}

Based on the generated data, the parameter estimation problem can be represented as a separable nonlinear optimization problem: 
\begin{equation*}
\begin{split}
&\min\limits_{\boldsymbol{a},\boldsymbol{c}}r(\boldsymbol{a},\boldsymbol{c})=\frac{1}{2}\sum\limits_{i=1}^{200}\left(y_i-c_1 e^{-a_2 x_i^2}\cos(a_3 x_i)-\right.\\&\quad\quad\left. c_2 e^{-a_1 x_i^2}\cos(a_2 x_i))-c_3 e^{-a_4 x_i^2}\sin(a_1 x_i)\right)^2.
\end{split}
\end{equation*}
\begin{figure}[htbp]
    \centering
        {\includegraphics[width=0.6\linewidth]{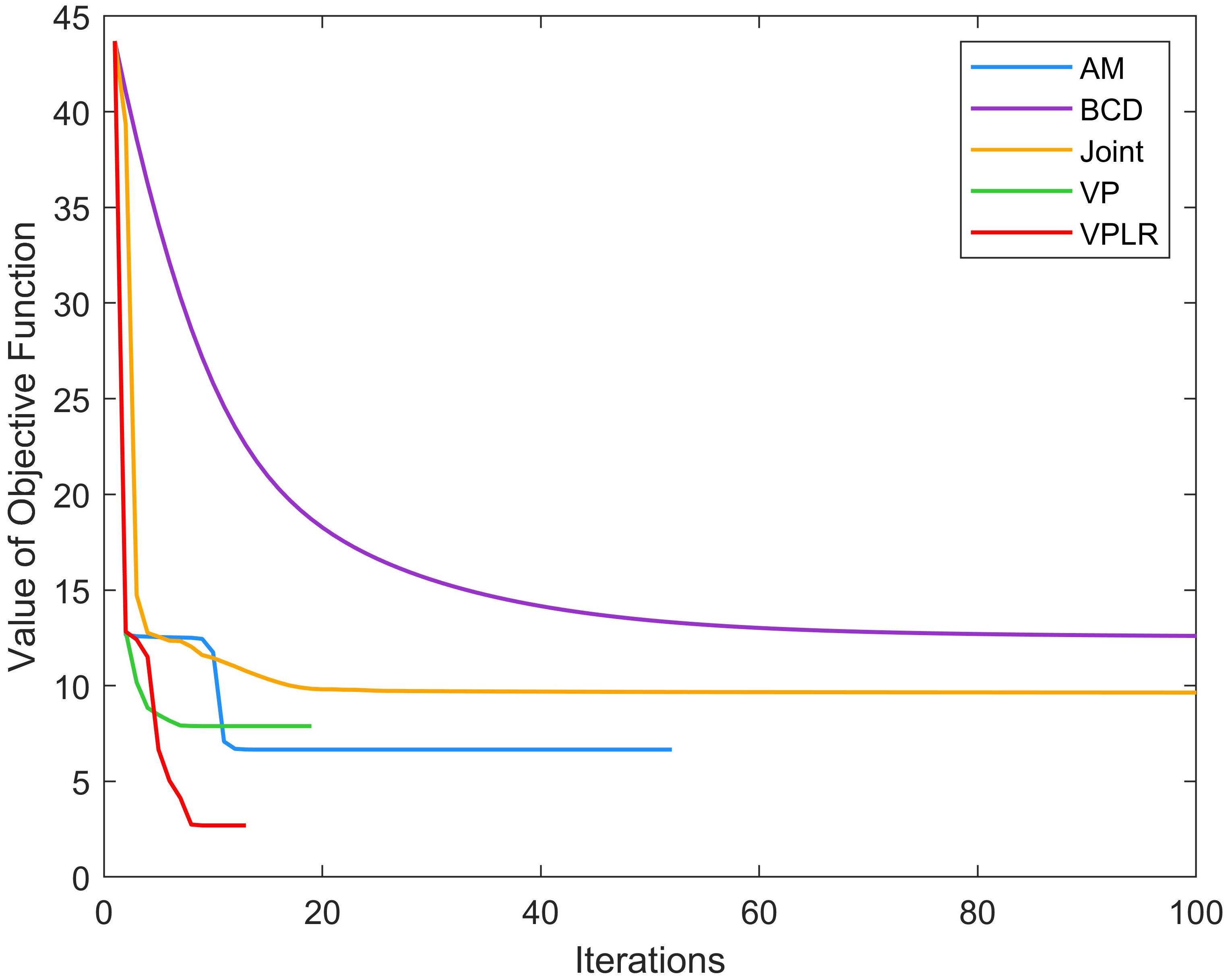}}
    \caption{Comparative convergence of different algorithms during iterative process.}
    \label{fig:4.2}
\end{figure}
We use different algorithms to estimate the model parameters, and the results are shown in Fig. \ref{fig:4.2} and \ref{fig:4.3}. \textcolor{black}{Notably, to verify the impact of large residual, we selected initial parameter values with large deviation from the true values and introduced significant noise, resulting in a large residual in the objective function}. As shown in Fig. \ref{fig:4.2}, both the traditional VP algorithm and the proposed VPLR algorithm properly handle the coupling relationship between the linear and nonlinear parameters in the model, {allowing for faster convergence}, while the joint algorithm and alternating algorithms (AM and BCD) converge relatively slower. Meanwhile, since the VPLR algorithm compensates for the effect of large residual on the Hessian matrix of the reduced objective function during the iteration process, it can find a better solution faster. This can also be observed from {Fig.} \ref{fig:4.3}, which shows the fitting results obtained by different algorithms. The models estimated using the Joint, BCD, and AM algorithms exhibit significant deviations in fitting the data, while the traditional VP algorithm, {hindered by} large initial deviations, {struggles} to fit the data adequately. In contrast, the model {obtained} by the VPLR algorithm fits the observed data exceptionally well. 
\begin{figure}[htbp]
    \centering
        {\includegraphics[width=0.6\linewidth]{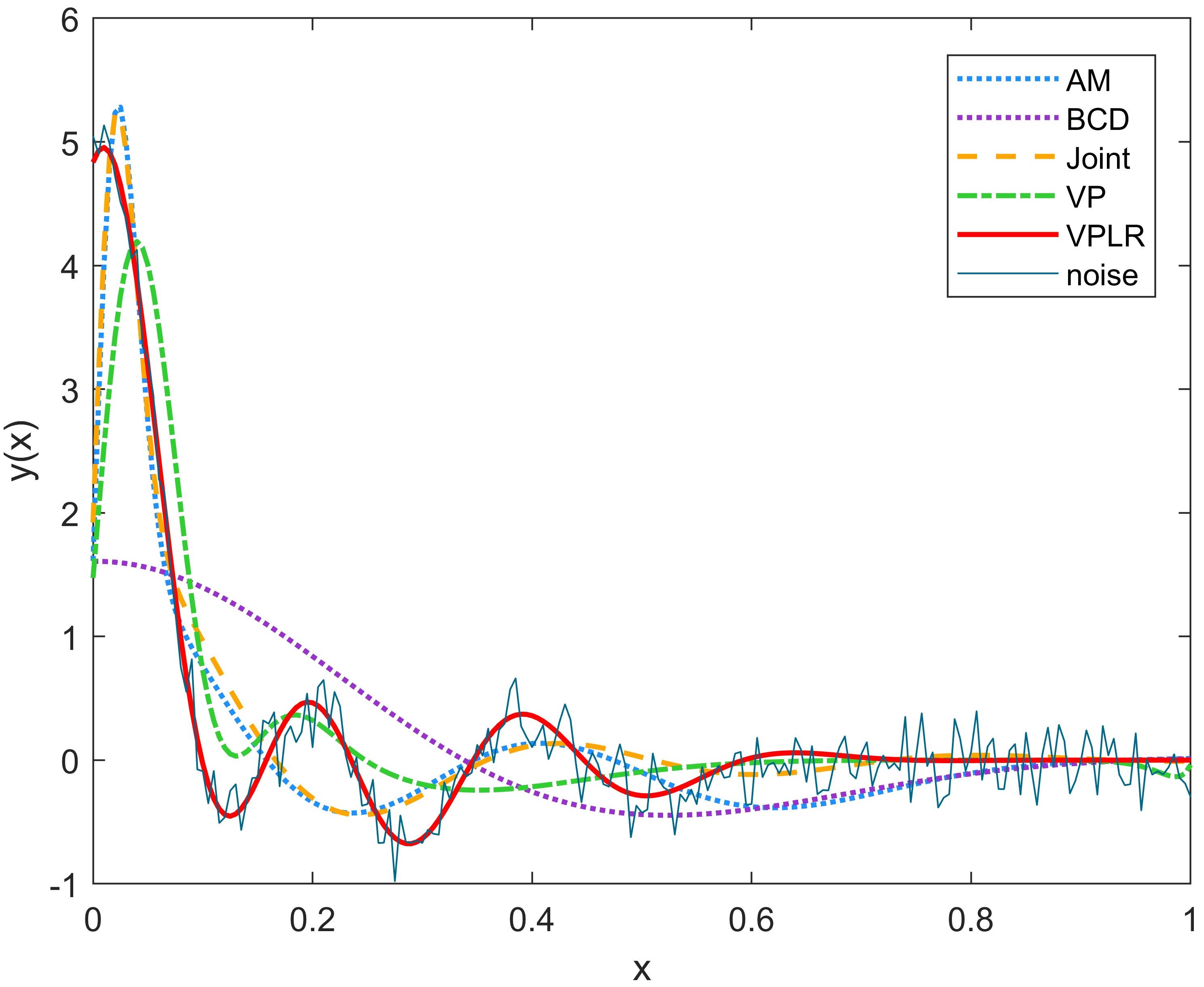}}
    \caption{Comparison of the results of fitting observed data to models estimated by different algorithms.}
    \label{fig:4.3}
\end{figure}







Additionally, to further verify the influence of initial values and observed samples on the convergence of different algorithms, we randomly generated 100 sets of randomized tests. The statistical results \textcolor{black}{of the objective function values and execution times} for fitting the exponential model using different algorithms are illustrated in Fig. \ref{fig:4.35} (AM: 2.70 ± 1.42, BCD: 14.02 ± 13.33, Joint: 3.20 ± 4.69, VP: 2.74 ± 5.16, VPLR: 2.67 ± 0.23) and \textcolor{black}{Fig. \ref{fig:4.36} (AM: 0.152 ± 0.088, BCD: 0.018 ± 0.007, Joint: 0.016 ± 0.0.10, VP: 0.032 ± 0.078, VPLR: 0.014 ± 0.026)}. 
\textcolor{black}{As shown in the figures, we can find that compared to the traditional VP and AM algorithms, VPLR converges more quickly to a better value and reduces overall execution times. While the BCD and Joint algorithms exhibit relatively low execution times, likely because they do not account for parameter coupling and do not require the solution of subproblems; however, however, this often leads to poorer convergence performance, as illustrated in Fig. 4.}

\begin{figure}[htbp]
    \centering
        {\includegraphics[width=0.6\linewidth]{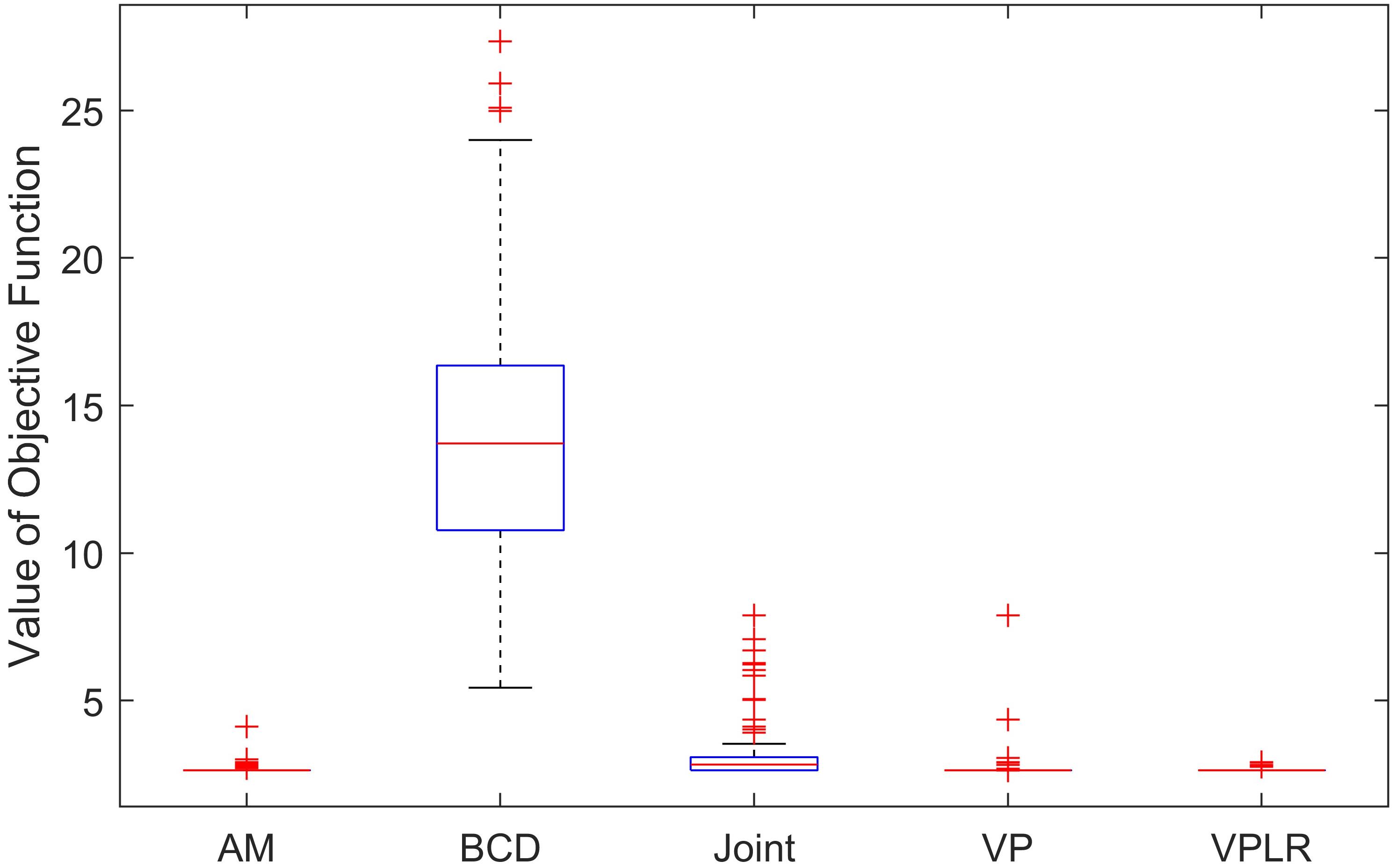}}
    \caption{Performance comparison of different algorithms with 100 sets of randomized tests.}
    \label{fig:4.35}
\end{figure}
\begin{figure}[htbp]
    \centering
        {\includegraphics[width=0.6\linewidth]{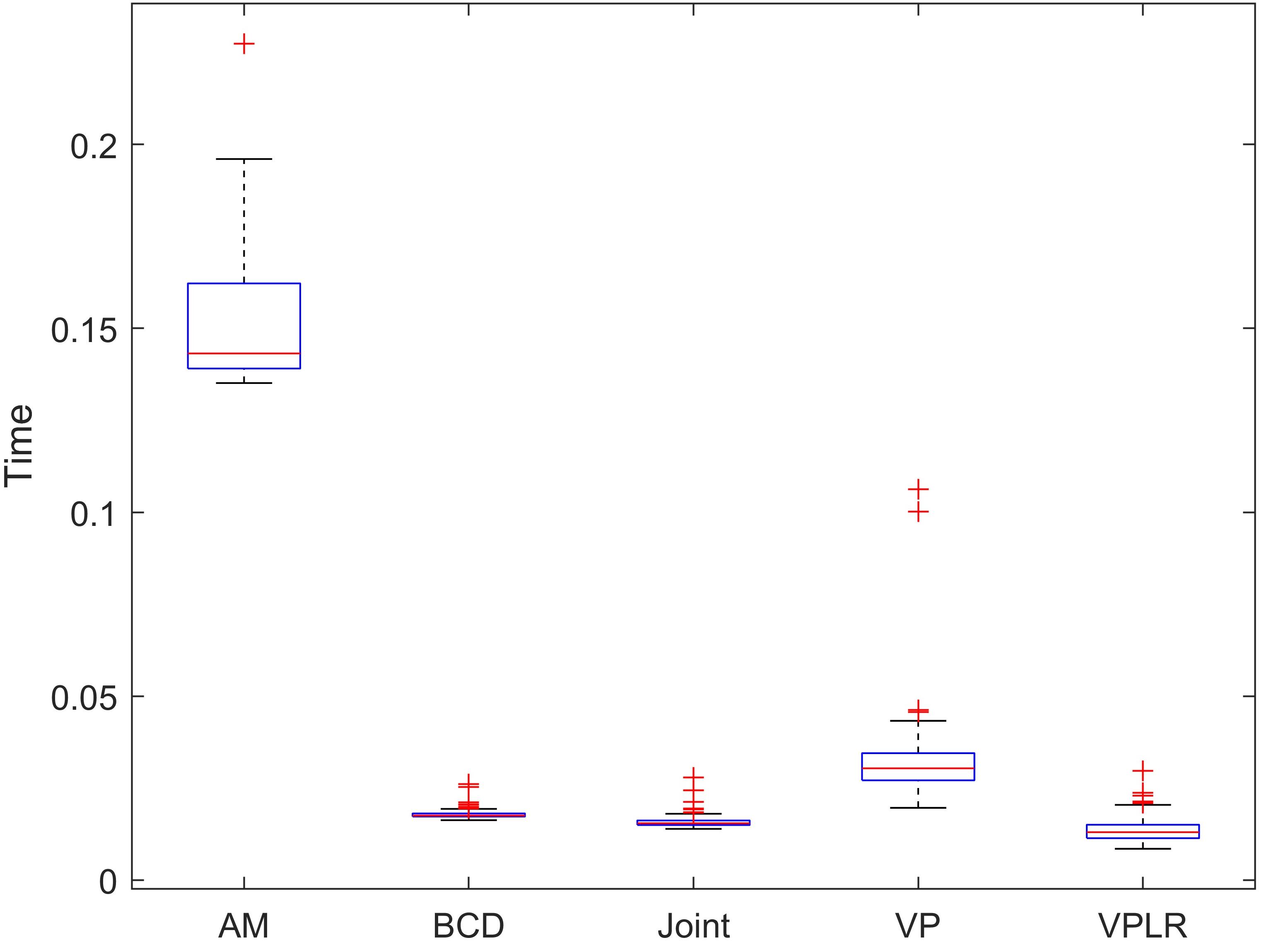}}
    \caption{{\color{black}Comparison of execution times across different algorithms using 100 sets of randomized tests.}}
    \label{fig:4.36}
\end{figure}

The comparison of the above experimental results demonstrates that the proposed VPLR algorithm, due to its ability to adeptly handle the coupling relationship between the parameters in the model and compensate for the effect of large residual on the Hessian matrix of the reduced objective function during the iteration process, is able to converge more quickly to a superior solution.

\subsection{Forecasting of Nonlinear Time Series Using the RBF-AR Model}
In this subsection, we {apply} the radial basis function network based autoregressive (RBF-AR) model to model a nonlinear time series. The RBF-AR($p, m, d$) model is a powerful statistical tool for system modeling \cite{peng2009nonlinear,peng2004rbf}, which can be expressed as:
\[\left.\left\{\begin{aligned}&y_t=\phi_0(\boldsymbol{x}_{t-1})+\sum_{i=1}^p\phi_i(\boldsymbol{x}_{t-1})y_{t-i}+e_t\\&\phi_i(\boldsymbol{x}_{t-1})=c_{i,0}+\sum_{j=1}^mc_{i,j}\exp\{-\lambda_j\|\boldsymbol{x}_{t-1}-\boldsymbol{z}_j\|^2\}\\&\boldsymbol{z}_k=(z_{k,1},z_{k,2},\cdots,z_{k,d})\\&\boldsymbol{x}_{t-1}=(y_{t-1},y_{t-2},\cdots,y_{t-d})\end{aligned}\right.\right.\]
where $p, m, d \in \mathrm{Z}^+$ denote the orders of the model; $\lambda_k$ and $\mathbf{z}_k$ $(k = 1,\cdots , m)$ represent the radius and centers of the RBF network, respectively; $\boldsymbol{x}_t$ encompasses certain explanatory variables within the system; $\phi(\cdot)$ denotes a nonlinear function and $e_t$ is the noise.

Next, we {use} the RBF-AR($8,1,3$) to fit the ozone column thickness data measured in the Arosa region of Switzerland. This dataset comprises 518 observations of the average thickness of the ozone columns. Following the procedure \textcolor{black}{outlined in} \cite{chen2018generalized}, we \textcolor{black}{preprocess} the raw data with a functional transformation to enhance the symmetry of the series and stabilize its variance{:}
\[ \tilde{y}_{i}=\ln(y_{i}-260).\]
The transformed data is shown in Fig. \ref{fig:4.4}.
\begin{figure}[htbp]
    \centering
        {\includegraphics[width=0.6\linewidth]{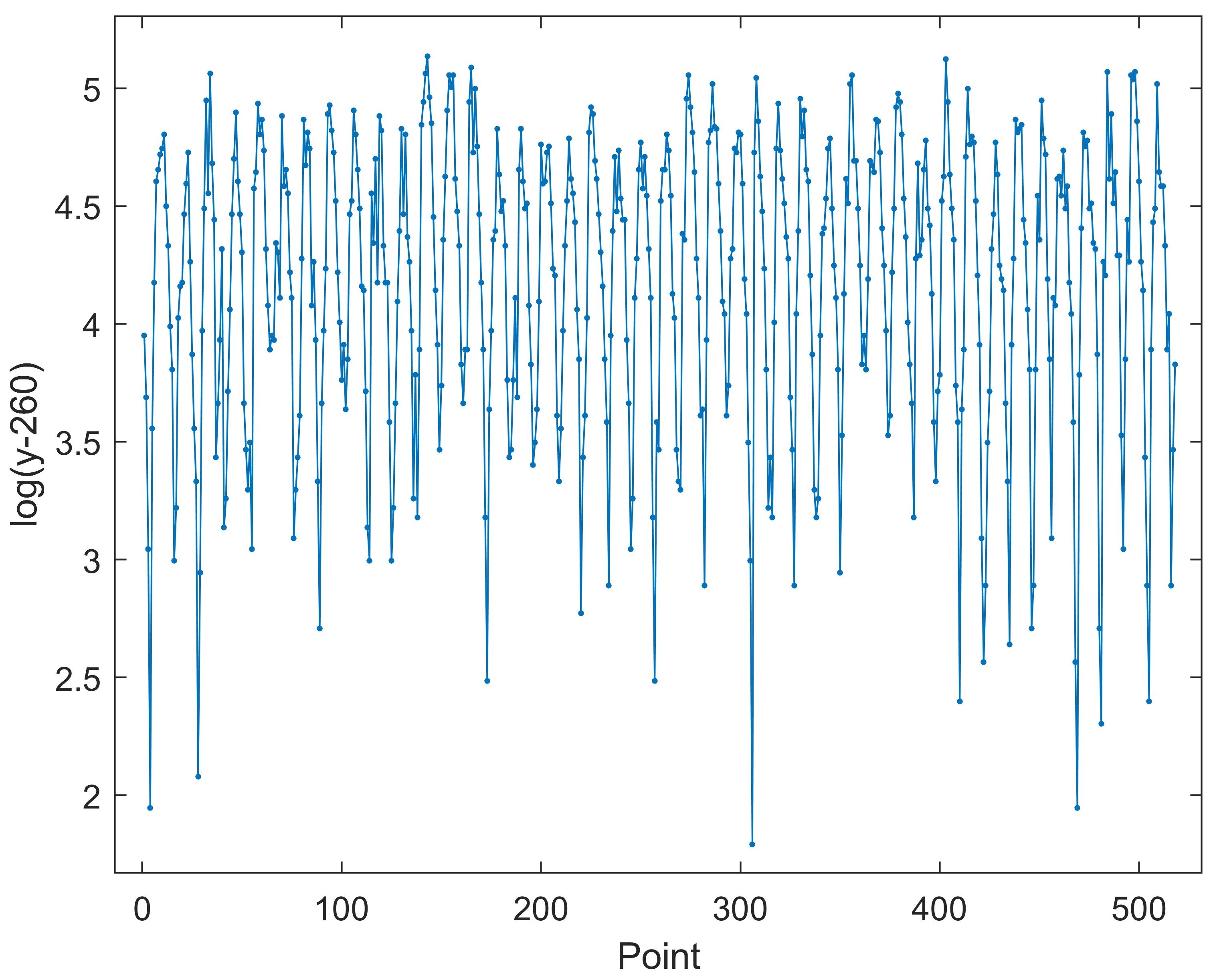}}
    \caption{Transformed ozone column data.}
    \label{fig:4.4}
\end{figure}
In this experiment, the {first} 450 data points are employed for model training, while the remaining data points  are served for \textcolor{black}{assessing the model's predictive performance.}  We adopt the Mean Squared Error (MSE) as the evaluation metric, which can be expressed as: 
\[\text{MSE}=\frac{1}{n}\sum_{i=1}^{n}(y_i-\hat{y}_i)^2,\]
where $\hat{y}_i$ is the output prediction.
\begin{figure}[htbp]
    \centering
        {\includegraphics[width=0.6\linewidth]{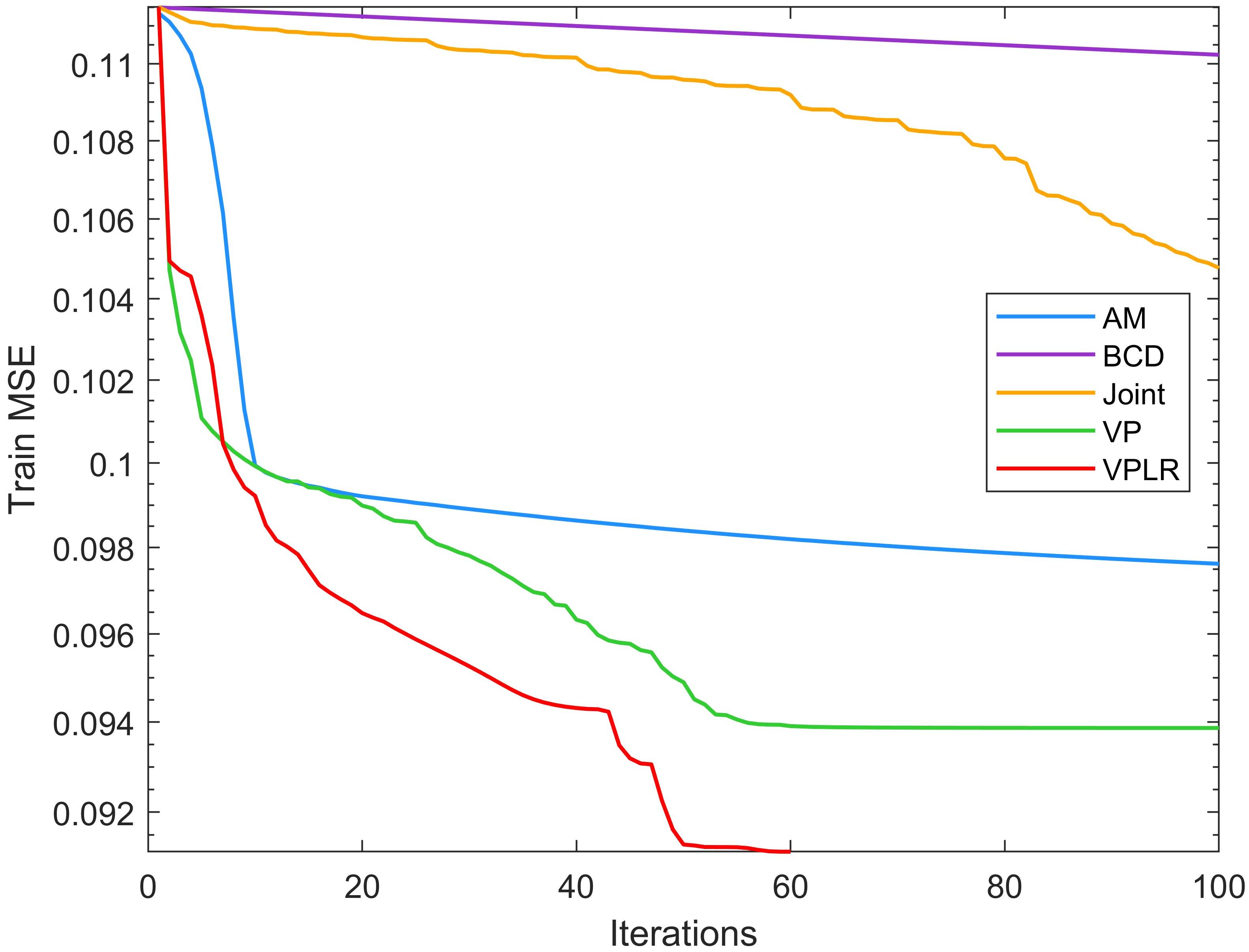}}
    \caption{Convergence process of different algorithms using the RBF-AR model to fit ozone column thickness data.}
    \label{fig:4.5}
\end{figure}
\begin{table}[htbp]
\caption{Comparison of training and testing errors of different algorithms for fitting ozone data sets using RBF-AR models.}
\setlength{\tabcolsep}{2.4mm}\renewcommand{\arraystretch}{1.8}
\begin{center}
{\begin{tabular}{cccccc}
\hline
\multicolumn{1}{l}{} & AM   & BCD    & Joint  & VP     & VPLR                        \\ \hline
Train MSE            & 0.098 & 0.110 & 0.105 & 0.094 & {$\mathbf{0.091}$} \\
Test MSE             & 0.178 & 0.214 & 0.200 & 0.170 & {$\mathbf{0.161}$} \\
Time             & 4.132 & 1.126 & 0.932 & 1.783 & {$\mathbf{0.857}$} \\
 \hline

\end{tabular}}
\end{center}
\label{tab:4.1}
\end{table}



As can be seen from Fig. \ref{fig:4.5}, the Joint, AM, and BCD algorithms converge slowly, particularly the Joint and BCD algorithms, as they ignore the separable structure of the RBF-AR model and the coupling relationship between the linear and nonlinear parameters {during} the optimization process. In contrast, the VPLR algorithm and the traditional VP algorithm can properly handle the coupling relationship between the parameters, enabling faster convergence. \textcolor{black}{Additionally,} by {compensating}  for the effect of large residual on the Hessian matrix of the reduced objective function, the VPLR algorithm not only converges more rapidly but also identifies more accurate solutions, thereby \textcolor{black}{enhancing} predictive performance. These are also clearly shown in Table \ref{tab:4.1}. 



\subsection{Fitting Concrete Data Using the RBF Network}
In the following part, we employ the radial basis function (RBF)
network
\[y=\sum_{k=1}^{m}c_{k}\exp(-r_{k}\left\|\boldsymbol{Y}-\boldsymbol{Z}\right\|^{2})\]
to model the concrete compressive strength data \cite{yeh1998modeling}.  \textcolor{black}{The dataset comprises} 1,030 observations, each containing eight features of concrete mix proportions and the corresponding response of concrete strength.

{Fig.} \ref{fig:4.6} shows the comparison of the convergence process of different algorithms for identifying RBF networks. From the figure, we can see that the VPLR algorithm \textcolor{black}{achieves} a faster convergence rate and identifies a solution that minimizes the fitting error \textcolor{black}{more effectively than other algorithms.} Tab. \ref{tab:4.2} presents the fitting errors, prediction errors \textcolor{black}{and execution times obtained} by different algorithms, providing a clearer comparison of their performance on both \textcolor{black}{training and testing sets.} Compared with other algorithms, the VPLR algorithm obtains the smallest training error, and superior predictive performance on the test set, {indicating} better generalization ability. 

These comparisons {demonstrate} the clear advantages of the  proposed algorithm in solving large-residual separable nonlinear optimization problems, with the condition of properly handling parameter coupling and compensating for the impact of residual. Moreover, the three experiments presented in this section reveal that the joint optimization algorithm exhibits slower local convergence than the VP algorithm, which is consistent with the theoretical analysis.

\begin{figure}[htbp]
    \centering
        {\includegraphics[width=0.6\linewidth]{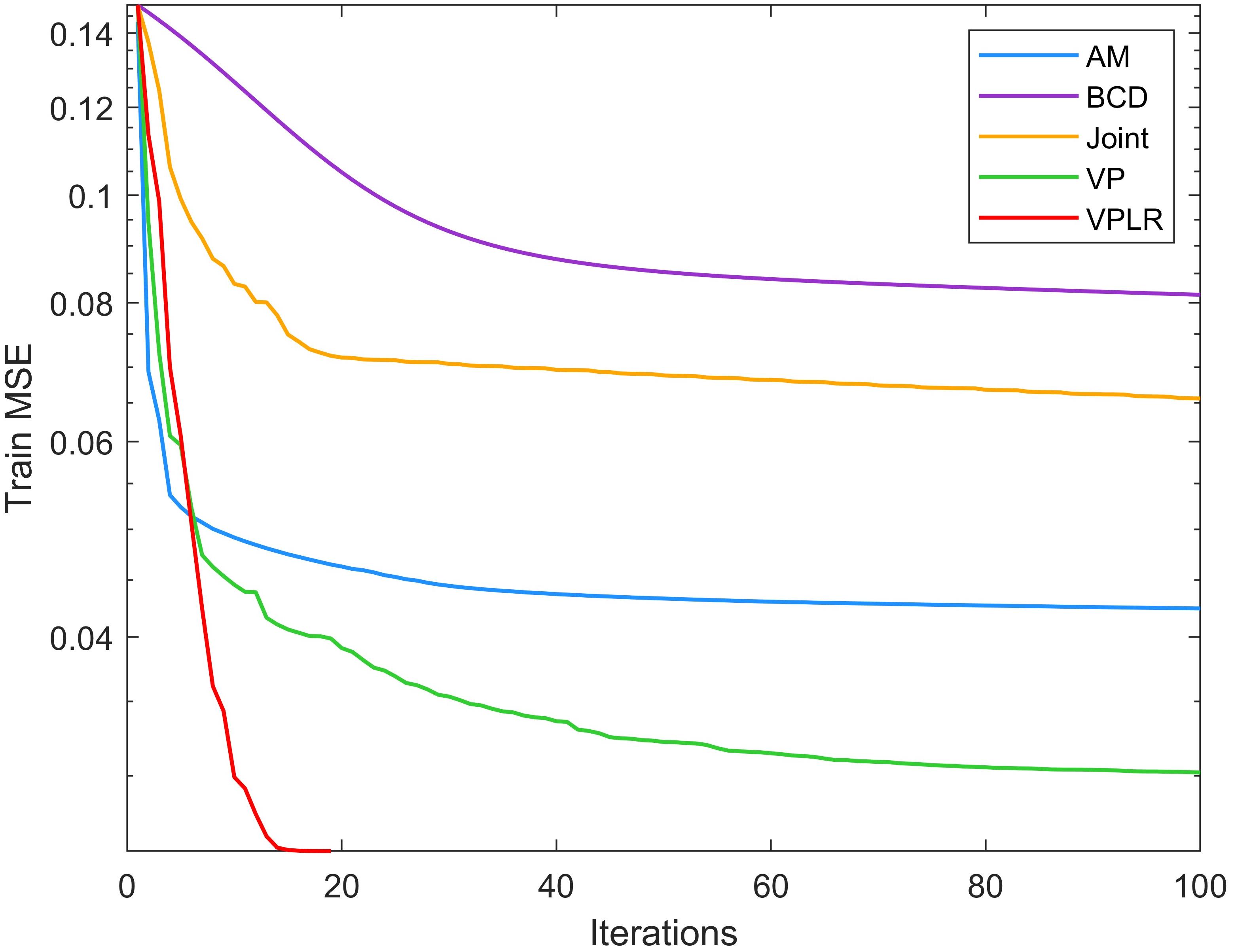}}
    \caption{Comparison of the convergence processes of different algorithms for concrete dataset.}
    \label{fig:4.6}
\end{figure}

\begin{table}[htbp]
\caption{Comparison of training and testing errors of different algorithms for Concrete data using RBF Networks.}
\setlength{\tabcolsep}{2.4mm}\renewcommand{\arraystretch}{1.8}
\begin{center}
{\begin{tabular}{cccccc}
\hline
\multicolumn{1}{l}{} & AM  & BCD  & Joint  & VP     & VPLR                        \\ \hline
Train MSE            & 0.042 & 0.081 & 0.066 & 0.030 & {$\mathbf{0.026}$} \\
Test MSE             & 0.059 & 0.103 & 0.071 & 0.085 & {$\mathbf{0.046}$} \\
Time             & 10.413 & 2.298 & 2.369 & 7.503 & {$\mathbf{1.785}$} \\
 \hline

\end{tabular}}
\end{center}

\label{tab:4.2}
\end{table}

\section{CONCLUSIONS}
\label{sec:5}
Separable nonlinear models are fundamental models in practical applications, such as system modeling, data analysis, and image processing. This paper investigates a highly efficient algorithm for identifying these models, namely the VP algorithm, and {establishes} a theoretical analysis framework for the impact of the reduced objective’s approximate Jacobian matrix (i.e., the approximate treatment of the coupling relationship between the linear and nonlinear parameters) on the convergence performance of the VP algorithm. A key theoretical contribution of this work is the demonstration that, under certain conditions, the Kaufman form of the VP algorithm can achieve similar local convergence rate as the Golub \& Pereyra form. Inspired by the theoretical analysis, we also design a refined VP algorithm for separable nonlinear optimization problem with large residual, denoted as VPLR, which improves the algorithm’s convergence by compensating the approximate Hessian matrix recursively while handling the coupling relationship between the model parameters. The theoretical research results of this paper fill the gap in this field and provide a strong support for understanding the mechanism of the VP algorithm and expanding its application scenarios.

\begin{thebibliography}{10}

\bibitem{saraf2022efficient}
Nilay Saraf and Alberto Bemporad.
\newblock An efficient bounded-variable nonlinear least-squares algorithm for embedded mpc.
\newblock {\em Automatica}, 141:110293, 2022.

\bibitem{chen2022modified}
Jing Chen, Manfeng Hu, Yawen Mao, and Quanmin Zhu.
\newblock Modified multi-direction iterative algorithm for separable nonlinear models with missing data.
\newblock {\em IEEE Signal Processing Letters}, 29:1968--1972, 2022.

\bibitem{li2021distributed}
Yafeng Li, Ju~H Park, Changchun Hua, and Guopin Liu.
\newblock Distributed adaptive output feedback containment control for time-delay nonlinear multiagent systems.
\newblock {\em Automatica}, 127:109545, 2021.

\bibitem{espanol2023variable}
Malena~I Espa{\~n}ol and Mirjeta Pasha.
\newblock Variable projection methods for separable nonlinear inverse problems with general-form tikhonov regularization.
\newblock {\em Inverse Problems}, 2023.

\bibitem{lin2017robust}
Zhouchen Lin, Chen Xu, and Hongbin Zha.
\newblock Robust matrix factorization by majorization minimization.
\newblock {\em IEEE Transactions on Pattern Analysis and Machine Intelligence}, 40(1):208--220, 2017.

\bibitem{sheng2015constrained}
Qiwei Sheng, Kun Wang, Thomas~P Matthews, Jun Xia, Liren Zhu, Lihong~V Wang, and Mark~A Anastasio.
\newblock A constrained variable projection reconstruction method for photoacoustic computed tomography without accurate knowledge of transducer responses.
\newblock {\em IEEE Transactions on Medical Imaging}, 34(12):2443--2458, 2015.

\bibitem{chung2023hybrid}
Julianne Chung, Jiahua Jiang, Scot~M Miller, and Arvind~K Saibaba.
\newblock Hybrid projection methods for solution decomposition in large-scale bayesian inverse problems.
\newblock {\em SIAM Journal on Scientific Computing}, pages S97--S119, 2023.

\bibitem{kovacs2019generalized}
P{\'e}ter Kov{\'a}cs, S{\'a}ndor Fridli, and Ferenc Schipp.
\newblock Generalized rational variable projection with application in ecg compression.
\newblock {\em IEEE Transactions on Signal Processing}, 68:478--492, 2019.

\bibitem{barbieri1992two}
Maria~M Barbieri and Piero Barone.
\newblock A two-dimensional prony's method for spectral estimation.
\newblock {\em IEEE Transactions on Signal Processing}, 40(11):2747--2756, 1992.

\bibitem{kovacs2022vpnet}
P{\'e}ter Kov{\'a}cs, Gerg{\H{o}} Bogn{\'a}r, Christian Huber, and Mario Huemer.
\newblock Vpnet: variable projection networks.
\newblock {\em International Journal of Neural Systems}, 32(01):2150054, 2022.

\bibitem{lederer2022cooperative}
Armin Lederer, Zewen Yang, Junjie Jiao, and Sandra Hirche.
\newblock Cooperative control of uncertain multi-agent systems via distributed gaussian processes.
\newblock {\em IEEE Transactions on Automatic Control}, 68(5):3091--3098, 2023.

\bibitem{newman2021train}
Elizabeth Newman, Lars Ruthotto, Joseph Hart, and Bart van Bloemen~Waanders.
\newblock Train like a (var) pro: Efficient training of neural networks with variable projection.
\newblock {\em SIAM Journal on Mathematics of Data Science}, 3(4):1041--1066, 2021.

\bibitem{golub2003separable}
Gene Golub and Victor Pereyra.
\newblock Separable nonlinear least squares: the variable projection method and its applications.
\newblock {\em Inverse Problems}, 19(2):R1, 2003.

\bibitem{tomasi1992shape}
Carlo Tomasi and Takeo Kanade.
\newblock Shape and motion from image streams under orthography: a factorization method.
\newblock {\em International Journal of Computer Vision}, 9:137--154, 1992.

\bibitem{hong2015secrets}
Je~Hyeong Hong and Andrew Fitzgibbon.
\newblock Secrets of matrix factorization: Approximations, numerics, manifold optimization and random restarts.
\newblock In {\em Proceedings of the IEEE International Conference on Computer Vision}, pages 4130--4138, 2015.

\bibitem{hu2023sar}
Fengyuan Hu, Xue Jiang, Junfeng Wang, and Xingzhao Liu.
\newblock Sar structure-from-motion via matrix factorization.
\newblock In {\em IGARSS 2023-2023 IEEE International Geoscience and Remote Sensing Symposium}, pages 6967--6970. IEEE, 2023.

\bibitem{erichson2020sparse}
N~Benjamin Erichson, Peng Zheng, Krithika Manohar, Steven~L Brunton, J~Nathan Kutz, and Aleksandr~Y Aravkin.
\newblock Sparse principal component analysis via variable projection.
\newblock {\em SIAM Journal on Applied Mathematics}, 80(2):977--1002, 2020.

\bibitem{dorabiala2024ensemble}
Olga Dorabiala, Aleksandr Aravkin, and J~Nathan Kutz.
\newblock Ensemble principal component analysis.
\newblock {\em IEEE Access}, 2024.

\bibitem{chen2023variable}
Guang-Yong Chen, Hui-Lang Xu, Min Gan, and CL~Philip Chen.
\newblock A variable projection-based algorithm for fault detection and diagnosis.
\newblock {\em IEEE Transactions on Instrumentation and Measurement}, 2023.

\bibitem{victor2022system}
St{\'e}phane Victor, Abir Mayoufi, Rachid Malti, Manel Chetoui, and Mohamed Aoun.
\newblock System identification of miso fractional systems: parameter and differentiation order estimation.
\newblock {\em Automatica}, 141:110268, 2022.

\bibitem{khatana2023dc}
Vivek Khatana and Murti~V Salapaka.
\newblock Dc-distadmm: Admm algorithm for constrained optimization over directed graphs.
\newblock {\em IEEE Transactions on Automatic Control}, 68(9):5365--5380, 2023.

\bibitem{hardt2014understanding}
Moritz Hardt.
\newblock Understanding alternating minimization for matrix completion.
\newblock In {\em 2014 IEEE 55th Annual Symposium on Foundations of Computer Science}, pages 651--660. IEEE, 2014.

\bibitem{yang2022proximal}
Yu~Yang, Qing-Shan Jia, Zhanbo Xu, Xiaohong Guan, and Costas~J Spanos.
\newblock Proximal admm for nonconvex and nonsmooth optimization.
\newblock {\em Automatica}, 146:110551, 2022.

\bibitem{kong2021kalman}
He~Kong, Mao Shan, Salah Sukkarieh, Tianshi Chen, and Wei~Xing Zheng.
\newblock Kalman filtering under unknown inputs and norm constraints.
\newblock {\em Automatica}, 133:109871, 2021.

\bibitem{1467459}
A.M. Buchanan and A.W. Fitzgibbon.
\newblock Damped newton algorithms for matrix factorization with missing data.
\newblock In {\em 2005 IEEE Computer Society Conference on Computer Vision and Pattern Recognition (CVPR'05)}, volume~2, pages 316--322 vol. 2, 2005.

\bibitem{wiberg1976computation}
T~Wiberg.
\newblock Computation of principal components when data are missing.
\newblock In {\em Proc. of Second Symp. Computational Statistics}, pages 229--236, 1976.

\bibitem{okatani2007wiberg}
Takayuki Okatani and Koichiro Deguchi.
\newblock On the wiberg algorithm for matrix factorization in the presence of missing components.
\newblock {\em International Journal of Computer Vision}, 72(3):329--337, 2007.

\bibitem{landeros2023mm}
Alfonso Landeros, Jason Xu, and Kenneth Lange.
\newblock Mm optimization: Proximal distance algorithms, path following, and trust regions.
\newblock {\em Proceedings of the National Academy of Sciences}, 120(27):e2303168120, 2023.

\bibitem{van2021variable}
Tristan Van~Leeuwen and Aleksandr~Y Aravkin.
\newblock Variable projection for nonsmooth problems.
\newblock {\em SIAM Journal on Scientific Computing}, (0):S249--S268, 2021.

\bibitem{golub1973differentiation}
Gene~H Golub and Victor Pereyra.
\newblock The differentiation of pseudo-inverses and nonlinear least squares problems whose variables separate.
\newblock {\em SIAM Journal on Numerical Analysis}, 10(2):413--432, 1973.

\bibitem{zhang2020offline}
Jize Zhang, Andrew~M Pace, Samuel~A Burden, and Aleksandr Aravkin.
\newblock Offline state estimation for hybrid systems via nonsmooth variable projection.
\newblock {\em Automatica}, 115:108871, 2020.

\bibitem{sjoberg1997separable}
Jonas Sjoberg and Mats Viberg.
\newblock Separable non-linear least-squares minimization-possible improvements for neural net fitting.
\newblock In {\em Neural Networks for Signal Processing VII. Proceedings of the 1997 IEEE Signal Processing Society Workshop}, pages 345--354. IEEE, 1997.

\bibitem{chen2021insights}
Guang-Yong Chen, Shu-Qiang Wang, Min Gan, and C.L.Philip Chen.
\newblock Insights into algorithms for separable nonlinear least squares problems.
\newblock {\em IEEE Transactions on Image Processing}, 30(2):1207--1218, 2021.

\bibitem{su2023nonmonotone}
Xiang-Xiang Su, Min Gan, Guang-Yong Chen, Lin Yang, and Jun-Wei Jin.
\newblock Nonmonotone variable projection algorithms for matrix decomposition with missing data.
\newblock {\em Pattern Recognition}, page 110150, 2023.

\bibitem{zeng2017regularized}
Xiaoyong Zeng, Hui Peng, and Feng Zhou.
\newblock A regularized snpom for stable parameter estimation of {RBF-AR} ({X}) model.
\newblock {\em IEEE Transactions on Neural Networks and Learning Systems}, 29(4):779--791, 2017.

\bibitem{laadjal2018online}
Khaled Laadjal, Mohamed Sahraoui, Antonio J~Marques Cardoso, and Ac{\'a}cio Manuel~Raposo Amaral.
\newblock Online estimation of aluminum electrolytic-capacitor parameters using a modified prony's method.
\newblock {\em IEEE Transactions on Industry Applications}, 54(5):4764--4774, 2018.

\bibitem{bock2021ecg}
Carl B{\"o}ck, P{\'e}ter Kov{\'a}cs, Pablo Laguna, Jens Meier, and Mario Huemer.
\newblock Ecg beat representation and delineation by means of variable projection.
\newblock {\em IEEE Transactions on Biomedical Engineering}, 68(10):2997--3008, 2021.

\bibitem{weber2024power}
Simon Weber, Je~Hyeong Hong, and Daniel Cremers.
\newblock Power variable projection for initialization-free large-scale bundle adjustment.
\newblock In {\em European Conference on Computer Vision}. Springer, 2024.

\bibitem{iglesias2023expose}
Jos{\'e}~Pedro Iglesias, Amanda Nilsson, and Carl Olsson.
\newblock expose: Accurate initialization-free projective factorization using exponential regularization.
\newblock In {\em Proceedings of the IEEE/CVF Conference on Computer Vision and Pattern Recognition}, pages 8959--8968, 2023.

\bibitem{kim2008training}
Cheol-Taek Kim and Ju-Jang Lee.
\newblock Training two-layered feedforward networks with variable projection method.
\newblock {\em IEEE Transactions on Neural Networks}, 19(2):371--375, 2008.

\bibitem{kaufman1975variable}
Linda Kaufman.
\newblock A variable projection method for solving separable nonlinear least squares problems.
\newblock {\em BIT Numerical Mathematics}, 15:49--57, 1975.

\bibitem{ruano1991new}
AEB Ruano, DI~Jones, and PJ~Fleming.
\newblock A new formulation of the learning problem of a neural network controller.
\newblock In {\em [1991] Proceedings of the 30th IEEE Conference on Decision and Control}, pages 865--866. IEEE, 1991.

\bibitem{song2020secant}
Xiongfeng Song, Wei Xu, Ken Hayami, and Ning Zheng.
\newblock Secant variable projection method for solving nonnegative separable least squares problems.
\newblock {\em Numerical Algorithms}, 85:737--761, 2020.

\bibitem{shearer2013generalization}
Paul Shearer and Anna~C Gilbert.
\newblock A generalization of variable elimination for separable inverse problems beyond least squares.
\newblock {\em Inverse Problems}, 29(4):045003, 2013.

\bibitem{ruhe1980algorithms}
Axel Ruhe and Per~{\AA}ke Wedin.
\newblock Algorithms for separable nonlinear least squares problems.
\newblock {\em SIAM review}, 22(3):318--337, 1980.

\bibitem{kelley1995iterative}
Carl~T Kelley.
\newblock {\em Iterative methods for linear and nonlinear equations}.
\newblock SIAM, 1995.

\bibitem{Gan2018On}
Min Gan, C.~L.~Philip Chen, Guang~Yong Chen, and Long Chen.
\newblock On some separated algorithms for separable nonlinear least squares problems.
\newblock {\em IEEE Transactions on Cybernetics}, 48(10):2866--2874, 2018.

\bibitem{aravkin2017efficient}
Aleksandr~Y Aravkin, Dmitriy Drusvyatskiy, and Tristan van Leeuwen.
\newblock Efficient quadratic penalization through the partial minimization technique.
\newblock {\em IEEE Transactions on Automatic Control}, 63(7):2131--2138, 2017.

\bibitem{dennis1996numerical}
John~E Dennis~Jr and Robert~B Schnabel.
\newblock {\em Numerical methods for unconstrained optimization and nonlinear equations}.
\newblock SIAM, 1996.

\bibitem{armijo1966minimization}
Larry Armijo.
\newblock Minimization of functions having lipschitz continuous first partial derivatives.
\newblock {\em Pacific Journal of mathematics}, 16(1):1--3, 1966.

\bibitem{nocedal1999numerical}
Jorge Nocedal and Stephen~J Wright.
\newblock {\em Numerical optimization}.
\newblock Springer, 1999.

\bibitem{guminov2021combination}
Sergey Guminov, Pavel Dvurechensky, Nazarii Tupitsa, and Alexander Gasnikov.
\newblock On a combination of alternating minimization and nesterov’s momentum.
\newblock In {\em International Conference on Machine Learning}, pages 3886--3898. PMLR, 2021.

\bibitem{nutini2022let}
Julie Nutini, Issam Laradji, and Mark Schmidt.
\newblock Let's make block coordinate descent converge faster: faster greedy rules, message-passing, active-set complexity, and superlinear convergence.
\newblock {\em Journal of Machine Learning Research}, 23(131):1--74, 2022.

\bibitem{wright2015coordinate}
Stephen~J Wright.
\newblock Coordinate descent algorithms.
\newblock {\em Mathematical Programming}, 151(1):3--34, 2015.

\bibitem{peng2009nonlinear}
Hui Peng, Jun Wu, Garba Inoussa, Qiulian Deng, and Kazushi Nakano.
\newblock Nonlinear system modeling and predictive control using the {RBF} nets-based quasi-linear {ARX} model.
\newblock {\em Control Engineering Practice}, 17(1):59--66, 2009.

\bibitem{peng2004rbf}
Hui Peng, Tohru Ozaki, Yukihiro Toyoda, Hideo Shioya, Kazushi Nakano, Valerie Haggan-Ozaki, and Masafumi Mori.
\newblock {RBF-ARX} model-based nonlinear system modeling and predictive control with application to a {NO}x decomposition process.
\newblock {\em Control Engineering Practice}, 12(2):191--203, 2004.

\bibitem{chen2018generalized}
Guang-Yong Chen, Min Gan, and Guo-long Chen.
\newblock Generalized exponential autoregressive models for nonlinear time series: Stationarity, estimation and applications.
\newblock {\em Information Sciences}, 438:46--57, 2018.

\bibitem{yeh1998modeling}
I-C Yeh.
\newblock Modeling of strength of high-performance concrete using artificial neural networks.
\newblock {\em Cement and Concrete Research}, 28(12):1797--1808, 1998.

\end{thebibliography}

\end{document}